\documentclass{amsart}

\usepackage{amsmath,latexsym,amscd,amsbsy,amssymb,amsfonts,amsthm,fleqn,leqno,
euscript, graphicx, color}

\newtheorem{thm}{Theorem}[section]
\newtheorem{pro}[thm]{Proposition}
\newtheorem{lem}[thm]{Lemma}
\newtheorem{cor}[thm]{Corollary}
\newtheorem{ex}[thm]{Example}
\newtheorem{que}[thm]{Question}

\newtheorem{df}[thm]{Definition}
 
\newtheorem{rem}[thm]{Remark}

\newcommand{\Int}{\mbox{{\rm int}}\,}
\newcommand{\cl}{\mbox{{\rm cl}}}
\newcommand{\st}{\mbox{{\rm st}}}
\newcommand{\pr}{\mbox{{\rm pr}}}
\newcommand{\id}{\mbox{{\rm id}}}

\newcommand{\St}{\mbox{{\rm St}}}
\newcommand{\aut}{\mbox{{\rm Aut}}}
\newcommand{\CL}{\mbox{{\rm CL}}}

\begin{document}

\title{On two methods of constructing compactifications of topological groups}\thanks{The study was carried out with the help of the Center of Integration in Science, Ministry of Aliyah and Integration, Israel,  8123461, and is supported by ISF grant 3187/24.}

\author{K.\,L.~Kozlov}
\address{Department of Mathematics, Ben-Gurion University of the Negev, Beer Sheva, Israel}
\email{kozlovk@bgu.ac.il}
\author{A.\,G.~Leiderman}
\address{Department of Mathematics, Ben-Gurion University of the Negev, Beer Sheva, Israel}
\email{arkady@math.bgu.ac.il}

\date{}

 \maketitle
 
\begin{abstract}

A classification of proper compactifications of topological groups is presented from the point of view of the possibility of continuous extensions of group operations. Ellis' method allows one to obtain all right topological semigroup compactifications of a group on which its left action on itself (by multiplication on the left) continuously extends from its equivariant compactifications. The presentation of group elements as graphs of maps in the hyperspace with the Vietoris topology allows one to obtain compactifications on which the involution and the left action of a group on itself extend from its equivariant compactifications.

 \medskip
 
Keywords: topological group, compactification, semigroup, involution, enveloping Ellis semigroup, binary relation, hyperspace

\medskip

AMS classification: Primary 22F50, 54D35 Secondary 54H15, 54E15
\end{abstract}

\section{Introduction} 

On a topological group $G$ there are algebraic structures: 

\medskip

(continuous) binary operation --- {\it multiplication} $\cdot: G\times G\to G$,

\medskip

(continuous) unary operation --- {\it involution} (a se{\rm l}f-inverse map $*:G\to G$ such that $*(g)=g^{-1}$ and $*(gh)=*(h)*(g)$, element $g^{-1}$ is an {\it inverse} of $g$), 

\medskip

multiplication on the left specifies (continuous) {\it left action} 
$$\alpha: G\times G\to G,\ \alpha (g , h)=gh.$$ 

multiplication on the left and involution specifies (continuous) {\it "right"\ action} 
$$\alpha^*: G\times G\to G,\ \alpha^* (g , h)=hg^{-1}.$$ 

It was A.~Weil who showed that a topological group $G$ is (locally) precompact in the right uniformity iff some (locally) compact group $\bar G$ contains $G$ as a dense topological subgroup ($\bar G$ is the Weil completion of $G$). Having a system of (continuous) homomorphisms to compact groups separating the elements of $G$, Bohr (non-proper) compactification of a topological group can be obtained (topological embedding is replaced with injectivity). I.\,M.~Gel'fand and D.~Ra\v{\i}kov showed that every locally compact Abelian group has Bohr compactification in this sense. Leaving only continuity from the topological embedding, a "general" Bohr compactification of a topological group is similarly defined. The Bohr compactification is intimately connected to the finite-dimensional unitary representation theory of topological groups.

WAP-compactification of a group $G$ corresponds (using Gelfand--Kolmogoroff theory, which establishes an order-preserving bijective correspondence between Banach unital subalgebras of $C^*(G)$ and compactifications of $G$) to the algebra of weakly almost periodic functions. It is the maximal semitopological semigroup (non-proper) compactification of $G$. WAP-compactifications are important in understanding the structure and representations of topological groups.  M.~Megrelishvili~\cite{Megr2001} showed that the WAP-compactification may not be a proper compactification of $G$.

In topological dynamics, where the action of a topological group $G$ on a compact space $X$ is examined, the concept of an enveloping semigroup is a crucial tool for studying orbit closures and, consequently, minimal flows. The enveloping semigroup of a dynamical system was introduced by R. Ellis~\cite{Ellis}. An enveloping semigroup of $G$ (see, for instance, \cite{Vries}) may be examined as a non-proper semigroup compactification $(S, \varphi)$ of $G$:  
\begin{itemize}
\item[(a)] $S$ is a right topological semigroup, 
\item [(b)] $\varphi: G\to S$ is a continuous homomorphism of semigroups and $\varphi (G)$ is dense in $S$, 
\item[(c)] the extension of left action of $G$ on itself to $S$ is continuous. 
\end{itemize}
In topological algebra, this approach led to the concept of a semigroup compactification of a semigroup $S$, which is a topological space. $(T, \varphi)$ is a semigroup compactification of $S$ if $T$ is a compact right topological semigroup, $\varphi: S\to T$ is a continuous homomorphism, $\varphi (S)$ is dense in $T$ and the extension of action of $S$ to $T$ is continuous on the left (see, for instance, \cite[Definition~4.7]{HS}).

The well-known examples of proper compactifications of a topological group $G$ are the maximal $G$-compactification $\beta_G G$ (the greatest ambit) and Roelcke compactification $b_r G$. They are Samuel compactifications of $G$ in the right uniformity and Roelcke uniformity on $G$, respectively.

\medskip

Considering a topological group $G$ as the transformation group of a compact space $X$ (action is effective), the elements of $G$ can be examined as self-maps of $X$. Two methods of obtaining compactifications of $G$ are possible. If the topology on $G$ coincides with the topology of pointwise convergence for the action 
$G\curvearrowright X$, then the Ellis' enveloping semigroup (the closure of $G$ in the compact space $X^X$) is the proper compactification of $G$. This method allows one to obtain a proper right topological semigroup compactification of $G$ on which the left action of $G$ on itself extends continuously. 

If the topology on $G$ coincides with the compact-open topology for the action $G\curvearrowright X$, then another method (the method of graphs) works. Represent elements of $G$ as graphs of maps in the hyperspace $2^{X\times X}$ with Vietoris topology and take the closure of $G$ in the compact space $2^{X\times X}$. The method of graphs allows one to obtain compactifications on which the involution,  the left and "right" action of $G$ on itself continuously extend~\cite{Kennedy} and~\cite{usp2001}.

\medskip

In the paper, by a proper compactification of a topological group $G$ we understand a pair $(b G, b)$, where $b G$ is a compact space and $b: G\to b G$ is a topological embedding. We show what proper compactifications of topological groups are obtained from equivariant compactifications of $G$ using the above two methods. This yields a possible classification of proper compactifications of topological groups from the point of view of the possibility of extensions of the noted above algebraic operations. A similar approach to (partial) classifications of non-proper compactifications of topological (semi-)groups can be found in~\cite[Ch.~3, \S\ 3.3]{BergJungMiln76}, \cite[Ch. 21]{HS} and~\cite{BITsankov}.

\medskip 

The possibility of applying the Ellis method and the method of graphs to obtain proper compactifications is justified in \S\ \ref{topolco}, where the coincidence of the topology of pointwise convergence and the compact-open topology for the extended action of a topological group $G$ on its $G$-compactification is established. 

In Section~\ref{class} (possible) classification of compactifications of topological groups is presented. Part I  (of the classification):
\begin{itemize}
\item compactifications to which the left action of $G$ on itself continuously extends ($G$-compactification),
\item compactifications to which the left and "right" actions of $G$ on itself continuously extend  ($G^{\leftrightarrow}$-compactification),
\item compactifications to which involution and the left action of $G$ on itself continuously extend ($G^*$-compactification).
\end{itemize}
Examining all $G$-compactifications of a topological group $G$, the application of the method of graphs provides proper compactifications on which the involution, the left and "right" actions of $G$ on itself continuously extend. In Theorem~\ref{thm2} the properties of the map of posets of compactifications corresponding to the method of graphs are described. It is a morphism of posets, all $G$-compactifications greater than or equal to the Roelcke compactification are sent to the Roelcke compactification. The image of $G^{\leftrightarrow}$-compactification is greater than or equal to itself. The characterization of $G^*$-compactifications of $G$ which are fixed points under the map of posets is given. The Roelcke compactification is the maximal element in the posets of $G^{\leftrightarrow}$- and $G^*$-compactifications of a topological group $G$. 

Part II (of the classification):
\begin{itemize}
\item $G$-compactifications which are right topological monoids (Ellis compactification),
\item $G$-compactifications which are semitopological monoids (sm-compactification).
\end{itemize}
Examining all $G$-compactifications of a topological group $G$, the application of the Ellis method provides Ellis compactifications. 
In Theorem~\ref{thm1}, the properties of the map of posets of compactifications corresponding to the Ellis method are described. It is an epimorphism of the poset of $G$-compactifications onto the poset of Ellis compactifications of $G$, the image of a $G$-compactification is greater than or equal to itself, and this epimorphism is idempotent. The maximal $G$-compactification $\beta_G G$ is the maximal element in the poset of Ellis compactifications of $G$. The WAP-compactification of $G$ is the maximal element in the poset of its sm-compactifications if it is not empty and does not exceed the Roelcke compactification $b_r G$. 

Part III of the classification specifies when compactifications are close to being a group. Compactifications that are compact semitopological inverse monoids with continuous inverses (sim-compactification). There exists a maximal element in the poset of sim-compactifciations of $G$  if it is not empty. In this case, it doesn't exceed the WAP-compactification of $G$ (it is also a proper compactification). 

\medskip

We adopt terminology and notation from~\cite{Engelking} and~\cite{RD}. Useful information about algebraic structures on topological spaces is in~\cite{ArhTk} and~\cite{HS}. Self-contained information on semigroups can be found in~\cite{CliffPrest} (see also \cite{Lawson}) and about semigroup of binary relations in~\cite{Birkhoff}. About Vietoris topology on a hyperspace, see~\cite{Beer}. 

All spaces considered are Tychonoff, and if necessary, we denote a (topological) space as $(X, \tau)$ where $\tau$ is a topology on (a set) $X$. $\Int A$ and $\cl\ A$ are internal and closure of a subset $A$ of space $X$ respectively. An abbreviation ``nbd''  refers to an open neighborhood of a point ($O_x$ is a nbd of a point $x$). The family of all nbds of the unit $e$ in $G$ is denoted $N_{G}(e)$. Maps are continuous maps. A proper compactification of $X$ is denoted $(bX, b)$ where $b X$ is compact, $b: X\to b X$ is a dense embedding. The map $f: b' X\to b X$ is the map of compactifications $(b' X, b')$, $(b X, b)$ if $b=f\circ b'$ and $b' X\geq bX$ in this case. $\id$ is the identity map.

Uniform space is denoted as $(X, \mathcal U)$, $u$ is a uniform cover and ${\rm U}$ is a corresponding entourage of $\mathcal U$~\cite{Isbell}, \cite{Engelking}. $\st (A, u)$ is the star of the set $A\subset X$ with respect to $u\in\mathcal U$. If the cover $v$ is a refinement of  $u$, we use the notation $v\succ u$. Uniformity on a topological space is compatible with its topology. $\mathcal U\vee\mathcal V$ denotes the least upper bound of uniformities $\mathcal U$ and $\mathcal V$. $(\tilde X, \tilde{\mathcal U})$ denotes the completion of $(X, \mathcal U)$. If $X$ is a subset of a uniform space $(Y, \tilde{\mathcal U})$, then $\mathcal U=\{\tilde u\wedge X=\{\tilde U\cap X\ |\ \tilde U\in\tilde u\}\ |\ \tilde u\in\tilde{\mathcal U}\}$, is a base of the {\it subspace uniformity} on $X$.

\begin{lem}\label{cont}
The map $f: (X, \mathcal U)\to (Y, \mathcal V)$ of uniform spaces is continuous {\rm(}in the induced topologies on $X$ and $Y$ by the correspondent uniformities{\rm)} at the point $x\in X$ iff $\forall\ v\in\mathcal V$  $\exists\ u\in\mathcal U$ such that $f(\st (x, u))\subset \st (f(x), v)$. 
\end{lem}

\begin{proof} 
Necessity. $\forall\ v\in\mathcal V$  the set $\Int (\st (f(x), v))$ is an open nbd of $f(x)$. There exist a nbd $O_x$ of $x$ such that $f(O_x)\subset\Int (\st (f(x), v))$ and $u\in\mathcal U$ such that $\st (x, u)\subset O_x$. Then $f(\st (x, u))\subset \st (f(x), v)$. 

Sufficiency. $\forall\ O_{f(x)}$ there exist $v\in\mathcal V$ such that $\st (f(x), v)\subset O_{f(x)}$ and $u\in\mathcal U$ such that $f(\st (x, u))\subset \st (f(x), v)$. Then $\Int (\st (x, u))$ is a nbd of $x$ and $f(\Int (\st (x, u)))\subset  O_{f(x)}$.
\end{proof}

\begin{rem}\label{denseent}
{\rm Let $A, B$ be subsets of a uniform space $(X, \mathcal U)$, $u, v\in \mathcal U$ are such that $\{\st (\st (x, v), v)\ |\ x\in X\}\succ\st (x, u)$. Then if $A'$ is a dense subset of $A$ and $A'\subset\st (B, v)$, then $A\subset\st (B, u)$. 

Indeed, $\forall\ x\in A$ $\exists\ a\in A'$ such that $x\in\st (a, v)$. For $a$ $\exists\ b\in B$ such that $a\in\st (b, v)$. Then $x\in\st (\st (b, v), v)\subset\st (b, u)\subset\st (B, u)$. }
\end{rem}


\section{Preliminaries} 

\subsection{$G$-spaces and self-inverse maps}

A $G$-space is a triple $(G, X, \theta)$, where $G$ is a topological group, $X$ is a space and $\theta: G\times X\to X$ is a left (continuous) action (abbreviation {\it $X$ is a $G$-space}). If $X$ is compact, then  $(G, X, \theta)$ is a {\it compact $G$-space}. 
The abbreviations $G\curvearrowright X$, $\theta (g, x)=gx$ for the action are used if the only action is used in the context. 
$$\theta (A, Y)=AY=\bigcup\{\theta (g, x)\ |\ g\in A\subset G,\ x\in Y\subset X\}.$$
A subset $Y$ of a $G$-space $X$ is an {\it invariant subset} if $GY=Y$. The restriction $\theta|_{G\times Y}$ of action $\theta$ to an invariant subset $Y$ is an action and  $(G, Y, \theta|_{G\times Y})$ is a $G$-space.

A subgroup $\St_x=\{g\in G\ |\ gx=x\}$ of $G$ is a {\it stabilizer} at a point $x\in X$. Action is {\it effective} if $e=\{g\in G\ |\ gx=x,\ \forall\ x\in X\}=\bigcap\{\St_x\ |\ x\in X\}$. Only $G$-spaces with effective actions are considered in the paper. 

A map $f: X\to Y$ of $G$-spaces $(G, X, \theta_X)$ and  $(G, Y, \theta_Y)$  is a {\it $G$-map} (or {\it equivariant map}) if $f\circ\theta_X=\theta_Y\circ (\id\times f)$.  If $f$ is not continuous, then it will be noted in the corresponding context.

\medskip

A map $s: X\to X$ is a {\it self-inverse} map if $s\circ s=\id$. 

\begin{df}
For spaces $X$ and $Y$ with self-inverse maps $s_X$ and $s_Y$ respectively, the map $f: X\to Y$ {\it commutes with self-inverse maps} if $$f\circ s_X=s_Y\circ f.$$ 
\end{df}

\begin{lem}\label{invaction}
If $(G, X, \theta)$ is a $G$-space with a self-inverse map $s$, then the {\it inverse action}  
$$\theta^* (g, x)=s(\theta (g, s(x))),\ g\in G,\ x\in X,$$
is correctly defined and $(\theta^*)^*=\theta$. 
\end{lem}

\begin{proof}
Indeed, 
$$\theta^* (e, x)=s(\theta (e, s(x)))=s(s(x))=x,\  x\in X,$$
$$\theta^* (hg, x)=s(\theta (hg, s(x)))=s(\theta (h, \theta (g, s(x))))=\theta^*(h, s(\theta (g, s(x))))=\theta^*(h, \theta^*(g, x)),$$
$g, h\in G$, $x\in X$. $\theta^*$ is continuous as a composition of continuous maps and 
$$(\theta^*)^*(g, x)=s(\theta^* (g, s(x)))=s(s(\theta (g, s(s(x)))))=\theta (g, x).$$
\end{proof}

\begin{lem}\label{ginversg}
Let $(G, X, \theta_X)$ and  $(G, Y, \theta_Y)$ be $G$-spaces with self-inverse maps $s_X$ and $s_Y$ respectively, $f: X\to Y$ is a not continuous $G$-map which commutes with self-inverse maps. Then $f$ is a not continuous $G$-map of $G$-spaces $(G, X, \theta^*_X)$ and $(G, Y, \theta^*_Y)$.
\end{lem}

\begin{proof}
$$f(\theta^*_X(g, x))=f(s_X(\theta_X (g, s_X(x))))=s_Y(f(\theta_X (g, s_X(x))))=$$
$$s_Y(\theta_Y (g, s_Y(f(x))))=\theta^*_Y(g, f(x)),\ g\in G,\ x\in X.$$
\end{proof}

A uniformity $\mathcal U_X$ on $X$ ($(G, X, \theta)$ is a $G$-space, $\mathcal U_X$ unduces the original topology on $X$) is called an {\it equiuniformity}~\cite{Megr1984} if the action $\theta$ is {\it saturated} (any homeomorphism from $G$ is a uniformly continuous map of $(X, \mathcal U_X$)) and is {\it bounded} (for any $u\in\mathcal U_X$ there are $O\in N_G(e)$ and $v\in\mathcal U_X$ such that $Ov=\{OV\ |\ V\in v\}\succ u$). In this case $(G, X, \theta)$ a {\it $G$-Tychonoff space} (the abbreviation $X$ is $G$-Tychonoff), the action is extended to the continuous action  $\tilde\theta: G\times \tilde X\to\tilde X$ on the completion $(\tilde X, \tilde{\mathcal U}_X)$ of $(X, \mathcal U_X)$ ($(G, \tilde X, \tilde\theta)$ is a $G$-space). The extension $\tilde{\mathcal U}_X$ of $\mathcal U_X$ to $\tilde X$ is an equiuniformity on $\tilde X$. The embedding $\jmath: X\to\tilde X$ is a $G$-map, $\tilde X$ or the pair $(\tilde X, \jmath)$  is a {\it $G$-extension of $X$} and $\jmath (X)$ is a dense {\it invariant} subset of $\tilde X$~\cite{Megr0}. $G$-spaces as extensions will also be denoted as $(G, (\tilde X, \jmath), \tilde\theta)$.

\begin{rem}\label{contact}
 {\rm If  for an action $G\curvearrowright X$ (not continuous) of a topological group $G$ on a space $X$ the uniformity $\mathcal U_X$ on $X$ is saturated and  bounded, then the the action $G\curvearrowright X$ is continuous.}
\end{rem}

\begin{lem}\label{exact}
Let $(G, X, \theta)$ be a $G$-space, $X$ is a dense subset of $Y$ and the continuous map $\tilde\theta: G\times Y\to Y$ is such that $\tilde\theta|_{G\times X}=\theta$. Then $Y$ is a $G$-extension of $X$.
\end{lem}

\begin{proof} Since $\tilde\theta$ is continuous,  $X$ is dense in $Y$  and $\tilde\theta|_{G\times X}=\theta$
$$\theta (e, x)=x,\ x\in X, \Longrightarrow \tilde\theta (e, y)=y,\ y\in Y,$$
$$\theta (gh, x)=\theta (g, \alpha_X (h, x)),\ x\in X, \Longrightarrow \tilde\theta (gh, y)=\tilde\theta (g, \tilde\theta (h, y)),\ y\in Y,\ g, h\in G.$$
Hence $(G, Y, \tilde\theta)$ is a $G$-space and $Y$ is a $G$-extension of $X$. 
\end{proof}

If $\mathcal U_X$ is  a totally bounded equiuniformity on $X$, then $(b X=\tilde X, \jmath_b)$ is a  {\it $G$-compactification}  or an {\it equivariant compactification} of  $X$. If $X$ is a $G$-Tychonoff space,  then the maximal totally bounded equiuniformity $p\mathcal U^{max}_X$ on $X$ exists and {\it the maximal $G$-compactification}  $(\beta_G X, \jmath_{\beta})$ of $X$ is the completion of $(X, p\mathcal U^{max}_X)$. The maximal totally bounded equiuniformity  $p\mathcal U_X^{max}$ on $X$ is the {\it precompact replica} of the maximal equiuniformity $\mathcal U_X^{max}$ on $X$ (see~\cite[Ch.~2]{Isbell}) and $\beta_G X$  is the {\it Samuel compactification of $(X, \mathcal U_X^{max})$} (see~\cite[Ch.\ 8, Problem 8.5.7]{Engelking}). 

\begin{lem}\label{invunif}
Let $(G, X, \theta)$ be a $G$-Tychonoff space with a self-inverse map $s$, $\mathcal U_X$ is an equiuniformity on $X$. Then 
$(G, X, \theta^*)$ is a $G$-Tychonoff space and 
$$\mathcal U^*_X=\{u^*=\{s(U)\ |\ U\in u\}\ |\ u\in\mathcal U_X\}$$
is an equiuniformity on $X$.
\end{lem}

\begin{proof} Since $s$ is a homeomorphism, it is easy to verify that $\mathcal U^*_X$ is a uniformity on $X$. Moreover, $(u^*)^*=u$, $ u\in\mathcal U_X$. 

$\mathcal U^*_X$ is saturated. Indeed, take $u^*\in\mathcal U^*_X$ and $g\in G$. Then for $u=(u^*)^*$ $\exists\ v\in\mathcal U_X$ such that 
$\{\theta (g, V)\ |\ V\in v\}\succ u$. Hence, 
$$\{\theta^* (g, s(V))=s(\theta (g, V))\ |\ V\in v\}\succ u^*$$
and $\{\theta^* (g, V)\ |\ V\in v^*\}\succ u^*$.

$\mathcal U^*_X$ is bounded. Indeed, take $u^*\in\mathcal U^*_X$. Then for $u=(u^*)^*$ $\exists\ v\in\mathcal U_X$ and $\exists\ O\in N_G(e)$ such that 
$\{\theta (O, V)\ |\ V\in v\}\succ u$. Hence, 
$$\{\theta^* (O, s(V))=s(\theta (O, V))\ |\ V\in v\}\succ u^*$$
and $\{\theta^* (O, V)\ |\ V\in v^*\}\succ u^*$. By Remark~\ref{contact} $(G, X, \theta^*)$ is a $G$-Tychonoff space.
\end{proof}

\begin{cor}
Let $(G, X, \theta)$ be a $G$-Tychonoff space with a self-inverse map $s$. 

If $\mathcal U_X=\mathcal U^*_X$, then $\mathcal U_X$ is an equiuniformity on $X$ for both actions $\theta$ and $\theta^*$.

If  $\mathcal U_X$ is an equiuniformity on $X$ for both actions $\theta$ and $\theta^*$, then $\mathcal U^*_X$ is  an equiuniformity on $X$ for both actions  $\theta$ and $\theta^*$.
\end{cor}


\subsection{Topological groups (uniformities, involution, left action)}
On a topological group $G$ four uniformities are well-known. The {\it right uniformity} $R$ (the base is formed by the uniform covers $\{Og=\bigcup\{hg\ |\ h\in O\}\ |\ g\in G\}$, $O\in N_G(e)$). The {\it left uniformity} $L$ (the base is formed by the uniform covers $\{gO=\bigcup\{gh\ |\ h\in O\}\ |\ g\in G\}$, $O\in N_G(e)$).  The {\it two-sided uniformity} $R\vee L$ (the least upper bound of the right and left uniformities). And the {\it Roelcke uniformity} $L\wedge R$ (the greatest lower bound of the right and left uniformities) (the base is formed by the uniform covers $\{OgO=\bigcup\{hgh'\ |\ h, h'\in O\}\ |\ g\in G\}$, $O\in N_G(e)$). A group  $G$ is {\it Roelcke precompact} if Roelcke uniformity is totally bounded. 

All the necessary information about these uniformities can be found in~\cite{RD}. 

{\it Roelcke compactification} $b_r G$ of $G$ is the Samuel compactification of $(G, L\wedge R)$, i.e. the completion of $(G, (L\wedge R)_{fin})$, where $(L\wedge R)_{fin}$ is the {\it precompact replica} of $L\wedge R$. If $G$ is Roelcke precompact, then $(L\wedge R)_{fin}=L\wedge R$ and {\it Roelcke compactification} $b_r G$ is the completion of $(G, L\wedge R)$.

\medskip 

On a topological group $G$ the left multiplication defines the {\it left action} $\alpha: G\times G\to G$, $\alpha (g , h)=gh$. The action $\alpha$ defines the right uniformity $R$ on $G$ with the base 
$$\{\{Ox=\alpha (g, x)\ |\  g\in O\}\ |\  x\in G\},\ O\in N_G(e).$$

\medskip 

{\it Involution on a topological group} $G$ is the self-inverse map $*:G\to G$, $*(g)=g^{-1}$, $g\in G$, and $*(gh)=*(h)*(g)$. Element $g^{-1}$ is called the  {\it inverse} of $g$.

By Lemma~\ref{invaction} the inverse action 
$$\alpha^* (g, h)=*(\alpha (g, *(h)))=hg^{-1},\ g, h\in G,$$
is defined.

The base of the left uniformity $L$ on $G$ can be defined as 
$$\{\{xO=\alpha^* (g, x))\ |\  g\in O^{-1}\}\ |\ x\in G\},\ O\in N_G(e),$$
and the base of the Roelcke uniformity $L\wedge R$ on $G$ can be defined as 
$$\{\{OxU=\alpha(h, \alpha^* (g, x)))\ |\  h\in O,\ g\in U^{-1}\}\ |\ x\in G\},\ O, U\in N_G(e).$$

\begin{pro}\label{unifcontRoelcke}
Let $(G, X, \theta)$ and $(G, X, \theta')$ be $G$-Tychonoff spaces, $\mathcal U_X$ is an equiuniformity on $X$ for both actions $\theta$ and $\theta'$ and $f: G\to X$ is a not continuous $G$-map of $G$-spaces $(G, G, \alpha)$,  $(G, X, \theta)$ and $G$-spaces $(G, G, \alpha^*)$,  $(G, X, \theta')$. 

Then $f: (G, L\wedge R)\to (X, \mathcal U_X)$ is uniformly continuous and, hence, continuous.
\end{pro}

\begin{proof}
Take an arbitrary $u\in\mathcal U_X$. Since $\mathcal U_X$ is an equiuniformity on $X$ (for actions  $\theta$ and  $\theta'$), there exist $O,\ W\in N_G(e)$ and  $v\in\mathcal U_X$ such that 
$$\{\theta (O, V)\ |\ V\in v\}\succ u\ \mbox{and}\eqno{\rm(O)}$$
$$\{\theta' (W, x)\ |\ x\in X\}\succ v.\eqno{\rm(W)}$$ 
For the cover $\{OgW=\alpha (O, \alpha^*(W, g))\ |\ g\in G\}\in L\wedge R$ one has 
$$f(OgW)=\theta (O, \theta' (W, f(g)))\stackrel{x=f(g)}{=}\theta (O, \theta' (W, x)).$$
By (W) $\exists\ V\in v\ \mbox{such that}\  \theta' (W, x)\subset V$. By (O) $\exists\ U\in u\ \mbox{such that}\  \theta (O, V)\subset U$. 
Therefore, $\forall\ g\in G$ $\exists\ U\in u$ such that $f(OgW)\subset U$. Hence, $f: (G, L\wedge R)\to (X, \mathcal U_X)$ is uniformly continuous. 
\end{proof}

From Proposition~\ref{unifcontRoelcke} and Lemma~\ref{ginversg} it follows.

\begin{cor}\label{cor1}
Let $(G, X, \theta)$ be a $G$-space with self-inverse map $s$, $f: G\to X$ is a not continuous $G$-map of $G$-spaces $(G, G, \alpha)$ and $(G, X, \theta)$ which commutes with involution on $G$ and $s$ on $X$. 
Then $f$ is a not continuous $G$-map of $G$-spaces $(G, G, \alpha^*)$ and  $(G, X, \theta^*)$.

If, additionally, $X$ is a $G$-Tychonoff space and $\mathcal U_X$ is an equiuniformity on $X$ for actions  $\theta$ and  $\theta^*$, then $f: (G, L\wedge R)\to (X, \mathcal U_X)$ is uniformly continuous and, hence, continuous.
\end{cor}


\subsection{Compact-open topology and topology of pointwise convergence}\label{topolco} 

For an action $G\curvearrowright X$ of a group $G$ on a topological space $X$ the subbase of {\it compact-open topology} $\tau_{co}$ on $G$ is formed by sets 
$$[K, O]=\{g\in G\ |\ gK\subset O\},\ \mbox{where}\ K\ \mbox{is a compact},\ O\ \mbox{is an open subset of}\ X.$$
If $X$ is compact, then $\tau_{co}$ is the least {\it admissible group topology} on $G$ ($(G, \tau_{co})$ is a topological group and the action $G\curvearrowright X$ is continuous)~\cite{Arens}. The notation {\it $((G, \tau_{co}), X, \theta)$ is a $G$-space} will mean that $\tau_{co}$ is a group topology on $G$ and the action $\theta$ is continuous.

For an action of a group $G$ on a topological space $X$ the subbase of {\it topology of pointwise convergence} $\tau_p$ on $G$ is formed by sets 
$$[x, O]=\{g\in G\ |\ gx\in O\},\ \mbox{where}\ x\in X,\ O\ \mbox{is an open subset of}\ X.$$
Evidently, $\tau_p\subset\tau_{co}$.

For a uniform space $(X, \mathcal L)$ an action $\theta: G\times (X, \mathcal L)\to  (X, \mathcal L)$ is {\it uniformly equicontinuous} if $\{\theta^g: X\to X,\ \theta^g(x)=\theta (g, x)\ |\ g\in G\}$  is a uniformly equicontinuous family of maps on $X$. Equivalently, for any $u\in\mathcal L$ there exists $v\in\mathcal L$  such that $gv=\{gV\ |\ V\in v\}\succ u$, for any $g\in G$. 
If an action  $G\curvearrowright (X, \mathcal L)$ is uniformly equicontinuous, then the topology of pointwise convergence $\tau_p$ is the least admissible group topology on $G$~\cite[Lemma 3.1]{Kozlov2022}. Moreover, $((G, \tau_p), X, \theta)$ is a $G$-Tychonoff space, see, for instance, \cite{Megr1984}. The notation  {\it $((G, \tau_{p}), X, \theta)$ is a $G$-space} will mean that $\tau_{p}$ is a group topology on $G$ and the action $\theta$ is continuous.

\begin{pro}\label{topcoex} 
Let $G=(G, \tau)$ be a topological group, $b G$ is a $G$-compactification of $G$ {\rm(}of $G$-space $(G, G, \alpha)$, $(G, b G, \tilde\alpha)$ is a $G$-space, $\tilde\alpha|_{G\times G}=\alpha${\rm)}, $G\subset X\subset b G$ is an invariant subset {\rm(}$(G, X, \tilde\alpha_X)$ is a $G$-space, $\tilde\alpha_X|_{G\times G}=\alpha${\rm)}. Then $((G, \tau_{co}), X, \tilde\alpha_X)$ and $((G, \tau_p), X, \tilde\alpha_X)$ are $G$-spaces and $\tau=\tau_{co}=\tau_p$. 
\end{pro}

\begin{proof} 
(i) $\tau=\tau^G_{co}=\tau^G_p$ if $X=G$ ($\tau^G_{co}$ and $\tau^G_p$ are compact-open topology and topology of pointwise convergence for the action $\alpha$). Indeed, the action  $\alpha: G\times G\to G$ is uniformly equicontinuous with respect to the left uniformity $L$ on $G$. Therefore, $\tau^G_p\subset\tau$ and, since $[e, O]=O$, $O\in N_G (e)$, $\tau^G_p=\tau$. 

Evidently, $\tau^G_p\subset\tau^G_{co}$. For any compact set $K$ and open set $O$ in $G$, $K\subset O$, there exist nbds $Ox_i$ of points $x_i\in F$ and nbds $V_i\in N_G(e)$, $i=1,\ldots, n$, such that $K\subset\bigcup\limits_{i=1}^n Ox_i$, $\alpha (V_i, Ox_i)\subset O$, $i=1,\ldots, n$. Then for $V=\bigcap\limits_{i=1}^n V_i\in N_G(e)$ one has $\alpha (V, K)\subset O$. Therefore, $V\subset [K, O]$ and $\tau^G_{co}\subset\tau^G_p=\tau$. Finally, $\tau=\tau^G_{co}=\tau^G_p$. 

\medskip

(ii) $\tau=\tau^{b G}_{co}=\tau^{b G}_p$ if $X=b G$ ($\tau^{b G}_{co}$ and $\tau^{b G}_p$ are compact-open topology and topology of pointwise convergence for the action $\tilde\alpha$). Indeed, $\tau^{b G}_{co}$ is the least admissible group topology on $G$ and $\tau^{b G}_{co}\subset\tau$, since $\tau$ is an admissible group topology on $G$ for the extended action $\tilde\alpha$. By (i) $\tau=\tau_p^G$ and $[x, O]=[x, \tilde O]$ for $x\in G$, $O\in\tau$, $\tilde O$ is open in $b G$ such that $\tilde O\cap G=O$. Hence, $\tau\subset\tau^{b G}_p\subset\tau^{b G}_{co}$ and finally, $\tau=\tau^{b G}_{co}=\tau^{b G}_p$. 

\medskip

(iii) Let $G\subset X\subset b G$ ($\tau^X_{co}$ and $\tau^X_p$ are compact-open topology and topology of pointwise convergence for the action $\tilde\alpha_X$). $\tau^X_p\subset\tau^X_{co}$. For any open set $[K, O]$, where $K$ is a compact and $O$ is an open subset of $X$, from the subbase of $\tau$ by (i), take an open in $b G$ set $\tilde O$ such that $\tilde O\cap X=O$. Then $[K, \tilde O]=[K, O]$ and the set $[K, \tilde O]$ is from the subbase of compact-open topology $\tau_{co}^{b G}$. Therefore, $\tau^X_{co}\subset\tau$ by (ii). 

$\tau=\tau_p^G\subset\tau^X_p$, hence, $\tau\subset\tau^X_p\subset\tau^X_{co}\subset\tau$ and finally,  $\tau=\tau^X_{co}=\tau^X_p$.
\end{proof}


\subsection{Hyperspace}\label{hyper}

Designations are from~\cite{Beer}. $\CL (X)$ is the family of nonempty closed subsets of $X$, $2^X=\CL (X)\cup\{\emptyset\}$. For $E\subset X$ 
$$E^+=\{F\in\CL (X)\ |\ F\subset E\},\ E^-=\{F\in\CL (X)\ |\ F\cap E\ne\emptyset\}.$$

\begin{df}~\cite[Definition 2.2.4]{Beer} 
Let $X$ be a Hausdorff space. The Vietoris topology $\tau_V$ on $\CL (X)$ has as a subbase all sets of the form $W^-$  and all sets 
of the form $W^+$, where $W$ is open in $X$. 
\end{df}

The Vietoris topology is Tychonoff if and only if $X$ is normal. If $X$ is compact, then $(\CL (X), \tau_V)$ is compact.  $(\CL (X), \tau_V)$ is a clopen subset of  $(2^X, \tau_V)$ and $\{\emptyset\}$ is an isolated point of $(2^X, \tau_V)$. 

A base for the Vietoris topology $\tau_V$ on $\CL (X)$ consists of all sets of the form 
$$[V_1, V_2, \ldots, V_n]=\{F\in\CL (X)\ |\ \forall\ i\leq n\ (F\cap V_i\ne\emptyset),\ F\subset\bigcup\limits_{i=1}^n V_i\},\eqno{\rm (V)}$$ 
where $V_1, V_2, \ldots, V_n$ is a finite family of open nonempty subsets of $X$. We use notation $2^X=(2^X, \tau_V)$ when $X$ is a compact space.

\medskip

Every homeomorphism $\varphi: X\to X$, where $X$ is a compact space, induces the homeomorphism $H_{\varphi}: 2^X\to 2^X$, $H_{\varphi}(A)=\varphi (A)$. 

\begin{rem}
{\rm If $s: X\to X$ is a self-inverse map,  $X$ is a compact space, then the map $H_s: 2^X\to 2^X$, $H_s(A)=s (A)$, is a self-inverse map.}
\end{rem}

\begin{pro}[action on a hyperspace]\label{actionhyper}  {\rm(}See, for instance, \cite[Remark 4.4]{DJPLP}{\rm).}
If $(G, X, \theta)$ is a compact $G$-space, then the {\it induced action} $\theta^H: G\times 2^X\to 2^X$, $\theta^H(g, A)=\theta (g, A)=\{gx\ |\ x\in A\}$, is continuous and $(G, 2^X, \theta^H)$ is a $G$-space.
\end{pro}

\begin{pro}\label{hyperinversemap} 
If $(G, X, \theta)$ is a compact $G$-space with self-inverse map $s$, then $(\theta^H)^*=(\theta^*)^H$.
\end{pro}

\begin{proof}
$$(\theta^H)^*(g, A)=H_s(\theta^H (g, s(A)))=H_s(\{\theta (g, s(x))\ |\ x\in A\})=\{s(\theta (g, s(x)))\ |\ x\in A\}=$$
$$=\{\theta^* (g, x)\ |\ x\in A\}=(\theta^*)^H (g, A),\ A\in 2^X,\ g\in G.$$
\end{proof}

\begin{pro}[map of hyperspaces]\label{maphyperF} 
Let $X$, $Y$ be compact spaces and $\varphi: Y\to X$ is an onto map. 

{\rm (A)} The map $H_{\varphi}: 2^X\to 2^Y$ is onto and perfect~\cite[Ch.\,3, \S\ 3.12, Problem 3.12.27 (e)]{Engelking}. 

\medskip

{\rm (B)} If $(G, X, \theta_X)$, $(G, Y, \theta_Y)$ are compact $G$-spaces and $\varphi$ is an onto $G$-map, then $H_{\varphi}$ is a $G$-map of compact $G$-spaces $(G, 2^X, \theta_X^H)$ and $(G, 2^Y, \theta_Y^H)$ . 

\medskip

{\rm (C)} If $\varphi$ commutes with self-inverse maps $s_X$ and $s_Y$ on $X$ and $Y$ respectively, then $H_{\varphi}$ commutes with self-inverse maps $H_{s_X}$ and $H_{s_Y}$. 

\medskip

{\rm (D)} If $G$-map $\varphi$ commutes with self-inverse maps $s_X$ and $s_Y$ on compact $G$-spaces  $(G, X, \theta_X)$, $(G, Y, \theta_Y)$ respectively, then $H_{\varphi}$ is a $G$-map of $G$-spaces $(G, 2^X,  (\theta_X^H)^*)$ and $(G, 2^Y, (\theta_X^H)^*)$. 
\end{pro}

\begin{proof} 
(B) $H_{\varphi} (\theta^H_X (g, A))=H_{\varphi}(\{\theta_X (g, x)\ |\ x\in A\})=\{\varphi(\theta_X (g, x))\ |\ x\in A\}=\{\theta_Y (g, \varphi (x))\ |\ x\in A\}=\{\theta_Y (g, y)\ |\ y\in \varphi (A)\}=\theta^H_Y (g, H_{\varphi}(A)),\ g\in G,\ A\in 2^X$.   

\medskip

(C) Let $\varphi\circ s_X=s_Y\circ\varphi$. Then 
$$H_{\varphi} (H_{s_X} (A))=H_{\varphi} (\{s_X(x)\ |\ x\in A\})=\{\varphi(s_X(x))\ |\ x\in A\}=\{s_Y(\varphi(x))\ |\ x\in A\}=$$
$$=H_{s_Y}(\{y\ |\ y\in \varphi (A)\})=H_{s_Y}(H_{\varphi}(A)),\ A\in 2^X.$$

(D) follows from (A) -- (C), Proposition~\ref{hyperinversemap} and Lemma~\ref{ginversg}.
\end{proof}

\medskip

{\it Uniformity on a hyperspace.} Let $(X, \mathcal U_X)$ be a uniform space. The family of sets
$$(\emptyset, \emptyset)\cup\{(A, B)\in\CL (X)\times\CL (X)\ |\ B\subset\st (A, u),\ \&\ A\subset\st (B, u)\},\ u\in\mathcal U_X,$$
is a base of uniformity $2^{\mathcal U_X}$ on $2^X$~\cite[Problem 8.5.16]{Engelking} or~\cite[Ch.\ 2,\ \S\ Hyperspace]{Isbell}.


\section{Compactifications of topological groups}\label{class}

Below $(G, G, \alpha)$ is a $G$-space where a topological group $G$ acts on itself by multiplication on the left. If $b G$ is $G$-compactification of $G$, then the extended action is $\tilde\alpha$, $(G, b G, \tilde\alpha)$ is a compact $G$-space. $\alpha^*$ is the inverse action to $\alpha$. $\tilde{\mathcal U}$ is the uniformity on $b G$, $\mathcal U$ is the subspace uniformity (equiuniformity) on $G$.

\subsection{$G$-compactifications of topological groups (extensions of actions and involution)}
\begin{df}
Let $G$ be a topological group. 
\begin{itemize}
\item A $G$-compactification $b G$ of $G$ is a $G$-compactification of $G$-space $(G, G, \alpha)$. 
\item A $G^{\leftrightarrow}$-compactification $b G$  of $G$ is a $G$-compactification of $G$-spaces $(G, G, \alpha)$ and $(G, G, \alpha^*)$. 
\item A $G^*$-compactification $b G$ of $G$ is a $G$-compactification of $G$ with a self-inverse map $s$ such that $s|_{G}=*$ {\rm(}the embedding of $G$ into $b G$ commutes with involution $*$ on $G$ and self-inverse map $s$ on $b G${\rm)}. 
\end{itemize}
\end{df}

\begin{pro}\label{compinv}
Let $G$ be a topological group. The following implications are valid for compactification $b G$ of $G$
$$G^*\mbox{-compactification}\longrightarrow G^{\leftrightarrow}\mbox{-compactification}\longrightarrow G\mbox{-compactification}.$$
\end{pro}

\begin{proof} 
Only the first implication is not evident. 

If $(b G, b)$ is a $G^*$-compactification of $G$, then the extension $\tilde\alpha$ of the action $\alpha$ is defined. Let $s$ be a self-inverse map on $b G$ such that $s|_{G}=*$ (equivalently, $b\circ *=s\circ b$).  By Lemma~\ref{invaction}  the inverse action 
$$\tilde\alpha^* (g, x)=s(\tilde\alpha (g, s(x))),\ g\in G,\ x\in b G,$$
on $b G$ is defined. By Lemma~\ref{ginversg} $b G$ is a $G$-compactification of $(G, G, \alpha^*)$. Hence, $b G$ is a $G^{\leftrightarrow}$-compactification of $G$. 
\end{proof}

$\mathbb{G} (G)$, respectively $\mathbb{G}^{\leftrightarrow} (G)$ and $\mathbb{G}^* (G)$ are posets of $G$-compactifications, respectively $G^{\leftrightarrow}$-compactifications and $G^*$-compactifications of $G$.

\begin{pro}\label{coinunif}
Let $b G\in\mathbb{G}^{\leftrightarrow}(G)$, $\mathcal U$ is the subspace uniformity on $G$. Then 
$$\mathcal U^{-1}=\{u^{-1}=\{U^{-1}\ |\ U\in u\}\ |\ u\in\mathcal U\},$$
is a totally bounded equiuniformity on $G$, $b' G\in\mathbb{G}^{\leftrightarrow}$, where $b' G$ is the completion of $(G, \mathcal U^{-1})$. 

$b^* G\in\mathbb{G}^*(G)$, where $b^* G$ is the completion of $(G, \mathcal U\vee\mathcal U^{-1})$. 
\end{pro}

\begin{proof}
The first two statements follow from Lemma~\ref{invunif} and total boundedness of $\mathcal U$. 

$\mathcal U\vee\mathcal U^{-1}$ is the least totally bounded uniformity on $G$ greater than or equal to $\mathcal U$ and $\mathcal U^{-1}$ and involution is uniformly continuous with respect to $\mathcal U\vee\mathcal U^{-1}$. Hence, $b^* G\in\mathbb{G}^*$. 
\end{proof}

\begin{pro}\label{blrlck}
$b_r G$ is the maximal element of $\mathbb{G}^{\leftrightarrow} (G)$ and $\mathbb{G}^* (G)$.
\end{pro}

\begin{proof}
$b_r G\in \mathbb{G}^* (G)\subset\mathbb{G}^{\leftrightarrow} (G)$. Indeed, $L\wedge R$ is an equiuniformity on a $G$-space $(G, G, \alpha)$, and involution is a uniform isomorphism of $(G, L\wedge R)$. Hence, the precompact replica $(L\wedge R)_{fin}$ of $L\wedge R$ is an equiuniformity on $G$, and involution is a uniform isomorphism of $(G, (L\wedge R)_{fin})$.

Let $(b G, b)\in\mathbb{G}^{\leftrightarrow} (G)$, $\tilde{\mathcal U}$ is the unique uniformity on $b G$. $\tilde{\mathcal U}$ is an equiuniformity for the extended actions $\tilde\alpha$ and ${\tilde\alpha}^*$ and $b$ is a $G$-map of $G$-spaces $(G, G, \alpha)$, $(G, b G, \tilde\alpha)$ and  $(G, G, \alpha^*)$, $(G, b G, {\tilde\alpha}^*)$. 

By Proposition~\ref{unifcontRoelcke} the $G$-map  $b: (G, L\wedge R)\to (b G, \tilde{\mathcal U})$ is uniformly continuous. Since $\tilde{\mathcal U}$ is totally bounded, the map $b: (G, (L\wedge R)_r)\to (b(G), \tilde{\mathcal U})$ is uniformly continuous. Hence, the extension $\tilde b: b_r G\to b G$ of $b$ is defined and $\tilde b$ is a map of compactifications (see~\cite[Lemma 3.5.6]{Engelking}). 
\end{proof}

\begin{rem}\label{examp36}
{\rm For SIN-groups ($L=R$~\cite[Proposition 2.17]{RD}) $G$-compactifications of a $G$-space $(G, G, \alpha)$ are $G$-compactifications of $G$-space $(G, G, \alpha^*)$. 

If $G$ is a SIN group, then $\mathbb{G}^{\leftrightarrow} (G)=\mathbb{G} (G)$, $\forall\ b G\in \mathbb{G} (G)$ $b G\leq b_r G$.}
\end{rem}

\begin{ex}\label{examp36}
{\rm $(\mathbb Z, +)$ is a discrete SIN-group and $\beta\mathbb Z=-\mathbb N^*\cup\mathbb Z\cup\mathbb N^*$, $\mathbb N^*=\beta\mathbb N\setminus\mathbb N$, $-\mathbb N^*=\{-p\ |\ p\in\mathbb N^*\}$, is the maximal $G$-compactification of $\mathbb Z$. $-\mathbb N^*$ is a compact invariant subset for the extended action. Therefore, $b\mathbb Z=\{-\infty\}\cup\mathbb Z\cup \mathbb N^*\in\mathbb G^{\leftrightarrow}(\mathbb Z)$. However, involution is not extended to $b\mathbb Z$. Therefore, $b\mathbb Z\in\mathbb{G}^{\leftrightarrow} (G)\setminus\mathbb G^*(\mathbb Z)$.}
\end{ex}


\subsection{Semigroup compactifications of topological groups, self-inverse maps and involutions}\label{seminv}

\begin{df} {\rm (see~\cite{CliffPrest}, \cite{Lawson})}
Let $(S, \bullet)$ be a semigroup {\rm(}multiplication $\bullet$ is a binary operation on $S${\rm)}. 

Monoid $M$ is a semigroup with identity $e$: $\forall\ x\in M\ (e\bullet x=x\bullet e=x)$.

\medskip

An involution on $S$ is a self-inverse map $s$ on $S$ such that $s(x\bullet y)=s(y)\bullet s(x)$, $\forall\ x, y\in S$.

A semigroup $S$ {\rm(}monoid $M${\rm)} with involution $s$ is called a semigroup {\rm(}monoid{\rm)} with involution.
\end{df}

\begin{df} {\rm(see~\cite{ArhTk})} 
A topological space $S$ and a semigroup $(S, \bullet)$ {\rm(}respectively monoid $M${\rm)} is 
\begin{itemize}
\item a right {\rm(}left{\rm)} topological semigroup {\rm(}respectively monoid $M${\rm)}, if $\forall\ s\in S$ the map $S\to S$   {\rm(}respectively $\forall\ s\in M$ the map $M\to M${\rm)},  $x\to x\bullet s$ {\rm(}$x\to s\bullet x${\rm)}, is continuous. 
\item a semitopological semigroup {\rm(}respectively monoid $M${\rm)}, if $S$  {\rm(}respectively $M${\rm)} is a right and left topological semigroup {\rm(}respectively monoid{\rm)}. 
\item a topological semigroup {\rm(}respectively monoid $M${\rm)}, if the map $S\times S\to S$ {\rm(}respectively $M\times M\to M${\rm)}, $(x, y)\to x\bullet y$, is continuous.  
\end{itemize}

If, additionally, $S$ is a semigroup {\rm(}respectively monoid $M${\rm)} with {\rm(}continous{\rm)} involution, then a semitopological / topological semigroup {\rm(}respectively monoid{\rm)} is called a  semitopological / topological semigroup  with involution {\rm(}respectively monoid  with involution{\rm)}.
\end{df}

\begin{rem}\label{seminvtop}
{\rm Let a space $X$ be a semigroup {\rm(}respectively monoid{\rm)} $(X, \bullet)$ with {\rm(}continuous{\rm)} involution $s$. Then the map  $\bullet_y: X\to X$, $\bullet_y(x)=x\bullet y$ is continuous iff the map $\bullet^{s(y)}: X\to X$, $\bullet^{s(y)}(x)=s(y)\bullet x$  is continuous.

Indeed, since $s(\bullet^{s(y)}(x))=s(s(y)\bullet x)=s(x)\bullet y=\bullet_y(s(x))$ and the involution $s$ is continuous, the map $\bullet_y: X\to X$ is continuous iff the map $\bullet^{s(y)}: X\to X$ is continuous. 

\medskip

If $X$ is a right {\rm(}left{\rm)} topological semigroup {\rm(}respectively monoid{\rm)} with {\rm(}continuous{\rm)} involution, then $X$ is a semitopological semigroup {\rm(}respectively monoid{\rm)}  with involution.}
\end{rem}

\begin{df}
Let $G$ be a topological group. 
\begin{itemize}
\item A left topological semigroup compactification $b G$ of $G$ {\rm(}abbreviation lts-compactification{\rm)} is 
a $G$-compactification of $G$ such that $(bG, \bullet)$ is a left topological semigroup and  $g\bullet h=gh$, $g, h\in G$ {\rm(}an embdding of $G$ into $b G$ is a homomorphism of semigroup $G$ into $b G${\rm)}.    
\item An Ellis compactification $b G$ of $G$ is a $G$-compactification of $G$ such that $(bG, \bullet)$ is a right topological monoid and $\bullet|_{G\times b G}=\tilde\alpha$~{\rm(}see {\rm\cite{KozlovSorin})}. 
\item A semitopological semigroup compactification $b G$ of $G$ {\rm(}abbreviation sm-compactification{\rm)} is a $G$-compactification of $G$ such that $(bG, \bullet)$ is a semitopological monoid and $g\bullet h=gh$, $g, h\in G$ {\rm(}an embdding of $G$ into $b G$ is a homomorphism of semigroup $G$ into $b G${\rm)}.  
\end{itemize}
\end{df}

\begin{rem}\label{tscomp}
{\rm (a) If $b G$ is a lts-compactification of $G$, then from continuity of maps $\tilde\alpha_g: b G\to b G$, $\tilde\alpha_g(x)=\tilde\alpha (g, x)$, $\bullet^g: b G\to b G$, $\bullet^g (x)=g\bullet  x$, $g\in G$, and their coincidence on dense subset $G$ of $b G$ it follows that $\tilde\alpha=\bullet |_{G\times b G}$. The same holds for sm-compactification.

\medskip

(b) $b G$ is a sm-compactification of $G$ iff  $b G$ is a compactification of $G$ and $(bG, \bullet)$ is a semitopological monoid such that  $\bullet|_{G\times b G}=\tilde\alpha$. Indeed, (necessity) if $(bG, \bullet)$ is a semitopological monoid such that $g\bullet h=gh$, $g, h\in G$, then the restriction $\bullet|_{G\times b G}$ is continuous~\cite[Propositions 6.1, 6.2]{JLawson} or~\cite[Theorem 2.38]{HS}.  $g\bullet h=\alpha (g, h)$, $g, h\in G$. Putting $\tilde\alpha (g, x)=g\bullet x$, $g\in G$, $x\in b G$, by Lemma~\ref{exact} continuous extension of the action $\alpha$ to $G\times b G$ is obtained. Hence, $\bullet|_{G\times b G}=\tilde\alpha$. Sufficiency is evident.

\medskip

(c) If $b G$ is a $G$-compactification of $G$ such that $(bG, \bullet)$ is a right topological semigroup and $\bullet|_{G\times b G}=\tilde\alpha$, then $(bG, \bullet)$ is a right topological monoid. Indeed, for the unit $e$ of $G$ one has $\bullet_{e}: b G\to b G$, $\bullet_{e}(x)=x\bullet e$, is a continuous map and $x\bullet e|_{G\times\{e\}}=\alpha|_{G\times\{e\}}$. Hence, $\bullet_{e}(x)=x$, $x\in G$, and continuity of $\bullet_{e}$ implies that $x\bullet e=x$, $x\in b G$. 

For the map $\bullet^{e}: b G\to b G$, $\bullet^{e}(x)=e\bullet x$, $\bullet^{e}(x)=\tilde\alpha (e, x)=x$, $x\in b G$. Hence, $e\bullet x=x$, $x\in b G$. 

The same observation is valid for sm-compactifications.

\medskip

(d) If $(b G, b)$ is an Ellis compactification of $G$, then condition $\bullet|_{G\times b G}=\tilde\alpha$ is equivalent to $b(g)\bullet x=\tilde\alpha (g, x)$, $g\in G$, $x\in b G$. Hence, $$b(gh)=b(\alpha (g, h))=\tilde\alpha (g, b(h))=b(g)\bullet b(h),\ g, h\in G$$ and $b$ is a homomorphism of semigroup $G$ into $b G$.

\medskip

If $b G$ is a $G$-compactification of $G$ such that $(bG, \bullet)$ is a topological monoid (equivalently semigroup) and $g\bullet h=\alpha (g, h)=gh$, $g, h\in G$, then $e$ is the identity in $b G$, the maximal subgroup $H(e)$ of $b G$ contaning $e$ as unit, contains $G$ and is dense in $b G$. By~\cite[Theorem 2.32]{HS} all maximal subgroups of compact topological semigroups are closed and are topological groups. Hence, $b G$ is a topological group.}
\end{rem}

\begin{que}
Is there an example of a topological group $G$ and its lts-compactification which is not a sm-compactification?
\end{que}

\begin{pro}\label{compsg}
Let $G$ be a topological group. The following implications are valid for compactification $b G$ of $G$
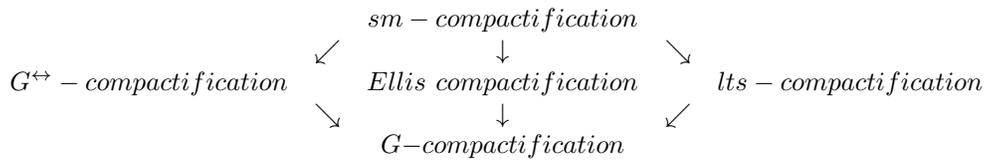
 $$\begin{array}{ccccc}
 &  & sm-compactification  &  &   \\ 
&  \swarrow & \downarrow   & \searrow &   \\
G^{\leftrightarrow}-compactification & & Ellis\ compactification  &   & lts-compactification \\
&  \searrow & \downarrow   &  \swarrow  &   \\
& & G{-}compactification  &    &  \\
\end{array}
$$
\end{pro}

\begin{proof}
Only implication $b G$ is a ss-compactification $\Longrightarrow$ $b G$ is a $G^{\leftrightarrow}$-compactification is not trivial.

If $(b G, \bullet)$ is a sm-compactification of $G$, then by~\cite[Propositions 6.1, 6.2]{JLawson} the restrictions $\bullet|_{G\times b G}$ and $\bullet|_{b G\times G}$ of $\bullet$ are continuous. $g\bullet h=\alpha (g, h)=\alpha^* (h^{-1}, g)$, $g, h\in G$. Putting $\tilde\alpha (g, x)=g\bullet x$, $\tilde{\alpha}^*(g, x)=x\bullet g^{-1}$, $g\in G$, $x\in b G$, by Lemma~\ref{exact} continuous extensions of actions $\alpha$ and $\alpha^*$ to $G\times b G$ are obtained. Hence, $b G$ is $G^{\leftrightarrow}$-compactification of $G$. 
\end{proof}

\begin{rem}
{\rm (a) $b G\leq b_r G$ for a sm-compactification $b G$ of $G$, $b G\leq b_r G$ for Ellis compactifications $b G$ of a SIN group $G$.

(b) Suppose the poset of sm-compactifications of $G$ is not empty. In that case, the WAP-compactification of $G$ is its maximal element and a proper compactification of $G$ (any non-proper semitopological semigroup compactification of $G$ can't be "greater" than its proper semitopological semigroup compactification).}
\end{rem}

$\mathbb{E} (G)$ is the poset of Ellis compactifications of $G$.

\begin{pro}\label{ssaieqinv}
If $b G$ is a sm-compactification of $G$ and $s$ is a {\rm(}continuous{\rm)} self-inverse map of $b G$ such that $s|_{G}=*$, then $s$ is an involution on $(b G, \bullet)$. 

If $b G$ is a lts-compactification of $G$ and $s$ is a {\rm(}continuous{\rm)} involuton on $(b G, \bullet)$, then $b G$ is a sm-compactification of $G$ {\rm(}with involution{\rm)}.

If $b G$ is an Ellis compactification of $G$ and $s$ is a {\rm(}continuous{\rm)} self-inverse map of $b G$ such that $s|_{G}=*$, then $b G$ is a sm-compactification of $G$ {\rm(}with involution{\rm)} iff $s$ is an involution on $(b G, \bullet)$.
\end{pro}

\begin{proof}
Since $s$ is a self-inverse map of $b G$ such that $s|_{G}=*$ 
$$s(g\bullet h)=s(gh)=h^{-1}g^{-1}=s(h)\bullet s(g),\ g, h\in G.$$
Since $b G$ is a sm-compactification of $G$ $s(g\bullet y)=s(\tilde\alpha (g, y))$, $y\in b G$, $g\in G$. 
For a fixed $g\in G$ the map $F_g: b G\to b G$, $F_g (y)=s(y)\bullet s(g)$, is continuous and $F_g|_G=s\circ \alpha|_{\{g\}\times G}$. Since $\tilde\alpha: G\times b G\to b G$ is continuous and $\tilde\alpha|_{\{g\}\times G}= \alpha|_{\{g\}\times G}$, 
$$s(g\bullet y)=s(y)\bullet s(g),\ y\in b G, g\in G.$$
Since $b G$ is a sm-compactification of $G$  the maps $T_y: b G\to b G$, $T_y(x)=s(x\bullet y)$ and $P_y: b G\to b G$, $P_y(x)=s(y)\bullet s(x)$ are  continuous and $T_y|_G=P_y|_G$. Hence, $s(x\bullet y)=T_y (x)=P_y (x)=s(y)\bullet s(x)$, $x, y\in b G$, and $s$ is involution on $(b G, \bullet)$. 

\medskip

By Remark~\ref{seminvtop} if $b G$ is a lts-compactification of $G$ with involution on $(b G, \bullet)$, then $b G$ is a sm-compactification of $G$ (with involution).

\medskip

The proof of the last statement follows from the previous considerations.
\end{proof}

\begin{ex}\label{example}
{\rm $(\mathbb Z, +)$ is a discrete group. Let $b\mathbb Z=\{-\infty\}\cup\mathbb Z\cup\{+\infty\}$ be a linearly ordered ($-\infty< x<+\infty$, $x\in\mathbb Z$) $G$-compactification of $\mathbb Z$. $b\mathbb Z$ with multiplication $x\bullet y=x+y$, $\pm\infty\bullet x=x\bullet\pm\infty=\pm\infty$, $x, y\in\mathbb Z$, $\pm\infty\bullet+\infty=+\infty$, $\pm\infty\bullet-\infty=-\infty$ is a monoind and $\bullet|_{\mathbb Z}=+$. Moreover, the map $s: b\mathbb Z\to b\mathbb Z$, $s(-\infty)=+\infty$, $s(x)=-x$, $x\in\mathbb Z$, $s(+\infty)=-\infty$ is a self inverse map, $s|_{\mathbb Z}$ is the involution on $\mathbb Z$. The map $\bullet_y$, $y\in b\mathbb Z$, is continuous,  the map $\bullet^{+\infty}$ is discontinuous at the point $-\infty$ ($+\infty\bullet x=+\infty$, $x\in\mathbb Z$, $+\infty\bullet-\infty=-\infty$, and $\bullet^{+\infty}$ is discontinuous at the point $-\infty$). Hence, $b\mathbb Z\in\mathbb E (\mathbb Z)\cap\mathbb G^*(\mathbb Z)$, $s$ is not an involution on $b\mathbb Z$
$$(-\infty=s(-\infty\bullet+\infty),\ s(+\infty)\bullet s(-\infty)=-\infty\bullet+\infty=+\infty),$$
$b\mathbb Z$ is not a sm-compactification of $\mathbb Z$.}
\end{ex}

\begin{df}
Let $G$ be a topological group. 
\begin{itemize}
\item  A lts-compactification $b G$ of $G$ with $s$ is a lts-compactification of $G$ with a {\rm(}continuous{\rm)}  self-inverse map $s$ such that $s|_G=*$.
\item An Ellis compactification $b G$ of $G$ with $s$ is an Ellis compactification of $G$ with a {\rm(}continuous{\rm)}  self-inverse map $s$ such that $s|_G=*$. 
\item A sm$^*$-compactification $b G$ of $G$ is a sm-compactification of $G$ such that $(bG, \bullet)$ is monoid with {\rm(}continuous{\rm)} involution. 
\end{itemize}
\end{df}

\begin{df}
A semigroup $S$ {\rm(}monoid $M${\rm)} is an inverse semigroup {\rm(}inverse monoid $M${\rm)} if 
$$\forall\ a\in S\ (M)\ \exists\ \mbox{the unque}\ a^*\in S\ (M)\ \mbox{such that}\ a\bullet a^*\bullet a=a\ \&\ a^*\bullet a\bullet a^*=a^*,$$
$a^*$ is an inverse of $a$.
\end{df}

\begin{rem}{\rm 
Every inverse semigroup  {\rm(}monoid{\rm)} is a semigroup {\rm(}monoid{\rm)} with involution $s: a\to a^*$ (see, for instance,~ \cite{Lawson}).}
\end{rem}

\begin{df}
Let $G$ be a topological group. 
\begin{itemize}
\item A sim-compactification $b G$ of $G$ is a $G$-compactification of $G$ and $(bG, \bullet)$ is a semitopological inverse monoid such that  $g\bullet h=gh$, $g, h\in G$, and inversion $a\to a^*$ is continuous.  
\end{itemize}
\end{df}

\begin{rem}
{\rm From considerations in~\cite[\S\ 2]{BITsankov} it follows that intermediate types (stronger than sm$^*$-compactifications and weaker than sim-compactifications) of compactifications of a topological group $G$ are not relevant.}
\end{rem}

\begin{pro}\label{order*comp}
Let $G$ be a topological group. The following implications are valid for compactification $b G$ of $G$
$$\begin{array}{ccccc}
 & &  sim{-}compactification  &  & \\
& & \downarrow   &  & \\
 &  & sm^*{-}compactification  &  \rightarrow   &  sm{-}compactification \\
 & \swarrow & \downarrow   &   & \downarrow \\
 G^*{-}compactification & & Ellis\ compactification\ with\ s  &  & G^{\leftrightarrow}-compactification \\
&  \searrow &  \downarrow &   \swarrow & \\
& & G{-}compactification &  & \\
\end{array}
$$
\end{pro}

\begin{rem}
{\rm (a) For every topo{\rm l}ogica{\rm l} group $G$ the comp{\rm l}etion of $G$ with respect to the maxima{\rm l} tota{\rm ll}y bounded equiuniformity is an E{\rm ll}is compactification $\beta_G G$ (it may be examined as the Samuel compactification of $(G, R)$ or {\it the greatest ambit of $G$}). $\beta_G G$ is the maximal element of $\mathbb{E}(G)$. If $G$ is a SIN group, then  $\beta_G G=b_r G$. 

\medskip

(b) A group $G$ has a compactification which is a topo{\rm l}ogica{\rm l} group iff $G$ is {\it precompact} (the right uniformity $R$ is tota{\rm ll}y bounded). Moreover, if $G$ is {\it precompact}, then $\beta_G G$ is the unique $G$-compactification of $G$~\cite{ChK2}. 

\medskip

(c) The topological group (($\aut ([0, 1]), \tau_p)$) has no sm-compactifications~\cite{Megr2001}.  

\medskip

(d) If $G$ is a locally compact group, then the Alexandroff one-point compactification $\alpha G$ is the least $G$-compactification of $G$ and $\alpha G$ is a sm$^*$-compactification~\cite{Berg}.

\medskip 

(e) If the poset of sm$^*$-compactifications (respectively sim-compactifications) of $G$ is not empty, then there is a maximal sm$^*$-compactification  (respectively sim-compactification) of $G$. 

Indeed, the product of (compact) monoids with involution (respectively inverse monoids) is a monoid with involution (respectively inverse monoid). The product in Tychonoff topology is a compact semitopological monoid with continuous involution (respectively compact semitopological inverse monoid with continuous inverse). The diagonal map of $G$ into the product is an embedding and an isomorphism of monoids. The closure of the image of $G$ is a compact semitopological monoid with continuous involution (respectively compact semitopological inverse monoid with continuous inverse). It is the maximal element in the corresponding poset.}
\end{rem}

\begin{que}
Is there an example of a topological group and its sm-compactification which is not a sm$^*$-compactification? 

Is there an example of a topological group and its compactification which is a semitopological inverse monoid with discontinuous inverse?

Let the poset of sm$^*$-compactifications of $G$ be not empty. Is its maximal element a WAP-compactification of $G$?
\end{que}

\begin{que}
Let a topological group have the unique $G$-compactification $b G$. Is $b G$ a compact topological group? 
\end{que}

\subsection{Maps of $G$-compactifications}

\begin{pro} 
Let $b' G$ and $b G$ be compactifications of a topological group $G$, $b' G\geq b G$ and $\varphi: b' G\to b G$ is the map of compactifications.
\begin{itemize}
\item[(a)] If $b G,\ b' G\in\mathbb{G}(G)$, then $\varphi$ is a $G$-map.
\item[(b)] If $b G,\ b' G\in\mathbb{G}^{\leftrightarrow}(G)$, then $\varphi$ is a $G$-map of $G$-spaces  $(G, b' G, \widetilde{\alpha}')$, $(G, b G, \widetilde{\alpha})$ and $(G, b' G, \widetilde{\alpha^*}')$, $(G, b G, \widetilde{\alpha^*})$.
\item[(c)]  If $b G,\ b' G\in\mathbb{G}^*(G)$, then $\varphi$  commutes with self-inverse maps $s'$ and $s$ on  $b' G$ and $b G$, respectively.
\item[(d)] If $b G,\ b' G\in\mathbb{E}(G)$, then $\varphi$ is a homomorphism of monoids.
\item[(e)]  If $b G,\ b' G$ are sm$^*$-compactifications of $G$, then $\varphi$ is a homomorphism of monoids with involutions. 
\end{itemize}
\end{pro}

\begin{proof}
(a) Let $(b' G, b')$ and $(b G, b)$ be $G$-compactifications of $G$, $(G, b' G, \tilde\alpha')$ and $(G, b G, \tilde\alpha)$ are correspondent compact $G$-spaces. 
If $\varphi: b' G\to b G$ is the map of compactifications, then $b=\varphi\circ b'$ and 
$$\varphi (\tilde\alpha' (g, b'(h)))=\tilde\alpha((\id\times\varphi) (g, b'(h)))=\tilde\alpha (g, b(h))\ g, h\in G.$$
Hence, $\varphi$ is a $G$-map (the maps $\varphi\circ\tilde\alpha'$ and $\tilde\alpha\circ (\id\times\varphi)$ are continuous and coincide on a dense subset $G\times G$ of $G\times b' G$).

\medskip

(b)  For the $G$-map $\varphi$ of $G$-compactifications $(b' G, b')$ and $(b G, b)$ (taking into account (a)), one must check the fulfillment of equality
$$\varphi (\widetilde{\alpha^*}'(g, x))=\widetilde{\alpha^*}(g, \varphi (x)),\ g\in G,\ x\in b' X,$$
where $(G, b' G, \widetilde{\alpha^*}')$ and $(G, b G, \widetilde{\alpha^*})$ are compact $G$-spaces.
Since $b=\varphi\circ b'$, for $g, h\in G$ one has 
$$\varphi (\widetilde{\alpha^*}'(g, b' (h)))=\varphi\circ b' (\alpha^*(g, h))=b (\alpha^*(g, h))=\widetilde{\alpha^*}(g, b (h)),\ g, h\in G.$$
The continuity of $\varphi$ and actions finishes. 

\medskip

(c)  For the $G$-map $\varphi$ of $G$-compactifications $(b' G, b')$ and $(b G, b)$ (taking into account (a)), one must check the fulfillment of equality
$$\varphi (s'(x))=s(\varphi (x)),\ x\in b' G.$$
Since $b=\varphi\circ b'$, for $g\in G$ one has 
$$\varphi (s'(b' (g)))=\varphi\circ b' (g^{-1}))=b (g^{-1})=s (b(g)).$$
The continuity of $\varphi$ and self-inverse maps finishes the proof. 

\medskip

(d) For the $G$-map $\varphi$ of $G$-compactifications $(b' G, b')$ and $(b G, b)$ (taking into account (a)), one must check the fulfillment of equality
$$\varphi(x\bullet y)=\varphi(x)\bullet\varphi(y),\ \forall\ (x, y)\in b' G\times b' G.\eqno{{\rm(eq)}}$$
Since $b=\varphi\circ b'$, where $b', b$ are homomorphisms of semigroups,  for $g, h\in G$ one has
$$\varphi(b'(g)\bullet b'(h))=\varphi(b'(g\cdot h))=b(g\cdot h)=b(g)\bullet b(h)=\varphi(b'(g))\bullet\varphi(b'(h))$$
and (eq) holds for $(x, y)\in b'(G)\times b'(G)$.

The restrictions of the maps $$\Psi: b' G\times b' G\to b G,\ \Psi (x, y)=\varphi(x\bullet y),\ \mbox{and}$$  
$$\Psi': b' G\times b' G\to b G,\ \Psi' (x, y)=\varphi(x)\bullet\varphi(y),$$
on the subset $b'(G)\times b' G$ are continuous. Hence, (eq) holds for $(x, y)\in b'(G)\times b' G$ (since continuous maps coincide on a dense subset $b'(G)\times b'(G)$).

Since $b' G$ is a right topological semigroup, for any fixed $y\in b' G$ the restrictions of the maps $\Psi$ and $\Psi'$ to $b' G\times\{y\}$ are continuous and coincide on a dense subset  $b'(G)\times\{y\}$. Hence, (eq) holds on $b' G\times\{y\}$, $y\in b' G$. Thus (eq) holds on $b' G\times b' G$. It remains to note that $b (e)$, $b'(e)$ are identities in $b G$ and $b' G$ respectively, and $\varphi (b' (e))=b (e)$. 

\medskip

(e) follows from (c) and (d). 
\end{proof}

The order on compactifications induces the order on subsets $\mathbb{G}(G)\supset\mathbb{G}^{\leftrightarrow}(G)\supset\mathbb{G}^*(G)$, $\mathbb{G}(G)\supset\mathbb{E}(G)\supset sm-{\rm compactifications\ of}\ G\supset  sm^*-{\rm compactifications\ of}\ G\supset  sim-{\rm compactifications\ of}\ G$. 

\begin{cor} 
$\mathbb{G}^* (G)$ is a partially ordered subset of $\mathbb{G}^{\leftrightarrow} (G)$, $\mathbb{G}^{\leftrightarrow} (G)$ is a partially ordered subset of $\mathbb{G} (G)$,

sim-compactifications is a partially ordered subset of sm$^*$-compactifications,  sm$^*$-compactifications is a partially ordered subset of sm-compactifications, sm-compactifications is a partially ordered subset of $\mathbb{E}(G)$, $\mathbb{E}(G)$ is a partially ordered subset of $\mathbb{G} (G)$.
\end{cor}


\section{Ellis compactifications}

Since the Cartesian product $X^X$ may be regarded as the set of self-maps of $X$ ($f=(f(x))_{x\in X}\in X^X$), there is a semigroup structure (multiplication is composition $\circ$ of maps) on $X^X$.

If $X$ is a topological space, then $(X^X, \circ)$ is a right topological monoid (multiplication on $X^X$ in the Tychonoff topology is continuous on the right, the identity map is identity in $X^X$).   

{\it The induced map of products.} The map $\varphi: b' X\to b X$ of compactifications of $X$ defines the equivalence relation $\sim_{\varphi}$ on $b' G$ 
$$x\sim_{\varphi}y\ \Longleftrightarrow\ \exists\ z\in b G\ \mbox{such that}\ x, y\in \varphi^{-1}(z).$$
By $[x]$ the equivalence class of $x\in b' G$ is denoted. 

The Cartesian product $\varphi^{b' X}:   (b' X)^{b' X}\to  (b X)^{b' X}$ of maps $\varphi$ is correctly defined.  

Let a subset $Y\subset  b' X$ consist of one representative from each equivalence class of $\sim_{\varphi}$.The map $\varphi_{Y}: (b X)^{b' X}\to  (b X)^{b X}$ is the composition of projection $\pr_{Y}: (b X)^{b' X}\to  (b X)^{Y}$ and the map $(b X)^{Y}\to (b X)^{b X}$ which identifies coordinates under the bijection $\varphi|_Y: Y\to X$. The composition $\Phi=\varphi_{Y}\circ\varphi^{b' X}:  (b' X)^{b' X}\to  (b X)^{b X}$ is correctly defined.  

\begin{pro}\cite[\S\ 2.7 Constructions,\ 2]{Vries}
Let $(G, X, \theta)$ be a $G$-space. Then $(G, X^X, \theta_{\Delta X})$ is a $G$-space, where 
$$\theta_{\Delta X} (g,  f)(x)=\theta (g, f(x)),\ x\in X,\eqno{(\Delta)}$$
is the diagonal action.
\end{pro}

If $\varphi: b' X\to b X$ is a map of $G$-compactifications of $X$, then the induced map $\Phi$ is a $G$-map ($\varphi^{b' X}$ and $\varphi_Y$ are $G$-maps). 

\begin{pro}
Let $(G=(G, \tau_p), X, \theta)$ be a $G$-space. Then 
\begin{itemize}
\item[(1)] the map
$$\imath_X: G\to X^X,\ \imath_X (g)(x)=\theta (g, x),\ x\in X,\eqno{(\imath)}$$
is a topological isomorphism of $G$ onto the subsemigroup of $X^X$.
\item[(2)]  $\theta_{\Delta X} (g,  f)=\imath_X (g)\circ f$,
\item[(3)]  $\imath_X$ is a $G$-map of $G$-spaces $(G, G, \alpha)$ and $(G, X^X, \theta_{\Delta X})$, the following diagram is commutative 
$$\begin{array}{ccc}
\quad G\times G & \stackrel{{\id\times \imath_X}}\hookrightarrow & G\times X^X \\
 \alpha  \downarrow &   & \quad \downarrow \theta_{\Delta X}\\
\quad G & \stackrel{\imath_X}\hookrightarrow & X^X.
\end{array}$$
\end{itemize}
\end{pro}

\begin{proof} (1) Since the action is effective, $\imath_X$ is injective. $\imath_X$ is an embedding, because topology on $G$ is topology of pointwise convergence. $\forall\ x\in X$ 
$$(\imath_X (g)\circ \imath_X (h))(x)=\imath_X (g)(\imath_X (h)(x))\stackrel{(\imath)}{=}\imath_X (g)(\theta (h, x)))\stackrel{(\imath)}{=}\theta (g, \theta(h, x)))=\theta (g\cdot h, x)\stackrel{(\imath)}{=}\imath_X (g\cdot h)(x).$$
Thus, $\imath_X (g\cdot h)=\imath_X (g)\circ \imath_X (h)$ and on $\imath_X (G)$ the stucture of a group (with multiplication $\circ$) is correctly defined ($\i_X (G)$ is a subsemigroup of $X^X$). 

(2) Let $f=(f(x))_{x\in X}$, $g\in G$.  $\forall\ x\in X$ 
$$\theta_{\Delta X} (g,  f)(x)\stackrel{(\Delta)}{=}\theta (g, f(x))\stackrel{(\imath)}{=}\imath_X (g) (f(x))=(\imath_X (g)\circ f)(x).$$
Thus, $\theta_{\Delta X} (g,  f)=\imath_X (g)\circ f$.

(3) $\forall\ g, h\in G$ 
$$\imath_X (g\cdot h)\stackrel{(1)}{=}\imath_X (g)\circ\imath_X (h)\stackrel{(2)}{=}\theta_{\Delta X}(g, \imath_X (h))=\theta_{\Delta X}((\id\times \imath_X) (g, h)).$$
\end{proof}

Since $\imath_X (G)$ is an invariant subset of the $G$-space $(G, X^X, \theta_{\Delta X})$, $\cl (\imath_X (G))$ is an invariant subset and an Ellis compactification of $G$ if $X$ is compact.

Define a map $\mathfrak{E}: \mathbb{G} (G)\to \mathbb{E} (G)$ ({\it Ellis map} of posets). If $b G\in\mathbb{G}(G)$ and $\imath_{b G}: G\to (b G)^{b G}$ is an embedding, then 
$$\mathfrak{E}(b G)={\rm cl} (\imath_{b G} (G)).$$

\begin{thm}\label{thm1}
The Ellis map $\mathfrak{E}:\mathbb{G}(G)\to \mathbb{E} (G)$ has the following properties
\begin{itemize}
\item $\mathfrak{E}$ is an epimorphism {\rm(}of posets{\rm)}, 
\item $\forall\ b G\in \mathbb{G} (G)$  $\mathfrak{E}(b G)\geq b G$, 
\item $\mathfrak{E}\circ\mathfrak{E}=\mathfrak{E}$.
\end{itemize}
\end{thm}

\begin{proof}
(i)  $\mathfrak{E}$ is morphism.  Let $(b' G, b')$, $(b G, b)$ be $G$-compactifications of $G$, $\varphi: b' G\to b G$ is the $G$-map of $G$-compactifications ($\tilde\alpha$ and $\tilde\alpha'$ are actions of $G$ on $b G$ and $b' G$ respectively). The $G$-map $\varphi$ induces the $G$-map  $\Phi:  (b' G)^{b' G}\to  (b G)^{b G}$ and the following diagram 
$$\begin{array}{ccl}
(b' G)^{b' G} & \stackrel{\Phi}{\longrightarrow} &  (b G)^{b G} \\
\imath_{b' G} \nwarrow &   &  \nearrow \imath_{b G} \\
 & G & 
\end{array}$$
is commutative.  In fact, for any  $h\in G$ 
$$\imath_{b G} (g)(b(h))=\tilde\alpha (g, b(h))=b(\alpha (g, h))\ \mbox{and}$$ 
$$[(\Phi\circ\imath_{b' G}) (g)] (b'(h))=\pr_{b'(h)}\big((\Phi\circ\imath_{b' G}) (g)\big)=\pr_{b'(h)}\big((\varphi_{Y}\circ\varphi^{b' G}\circ \imath_{b' G})(g)\big)=$$
\centerline{($\pr_{b'(h)}\circ\varphi_{Y}=\pr_{b(h)}$)}
$$=\pr_{b(h)}\big((\varphi^{b' G}\circ \imath_{b' G})(g)\big)=\varphi\big(\imath_{b' G} (g)(b'(h))\big)=\varphi\big(\tilde\alpha'(g, b'(h))\big)=$$
\centerline{($\varphi$ is a $G$-map)} 
$$=\tilde\alpha(g, \varphi (b'(h)))=\tilde\alpha(g, b(h))\stackrel{h\in G}=b(\alpha(g, h)).$$
Thus, $\imath_{b G} (g)(t)=(\Phi\circ\imath_{b' G}) (g)(t)$, $t\in b (G)$.  Being continuous they coincide on $b G$ and  $\imath_{b G} (g)=(\Phi\circ\imath_{b' G}) (g)$, $g\in G$.

Continuity of $\Phi$, commutativity of the above diagram ($\Phi|_{\imath_{b' G}}$ is a homeomorphism) and compactness of $\cl~\imath_{b' G} (G)$, $\cl~\imath_{b G} (G)$ yields that the restriction $\Phi|_{\mathfrak{E}(b' G)}: \mathfrak{E}(b' G)\to \mathfrak{E}(b G)$ is a map of $G$-compactifications. Hence, $\mathfrak{E}(b' G)\geq\mathfrak{E}(b G)$ and $\mathfrak{E}$ is a morphism of posets. 

\medskip

(ii) Let $(b G, b)$ be $G$-compactification of $G$. The projection $\pr_{e}: b G^{b G}\to b G=\{e\}\times b G$ is continuous and 
$$\pr_{b(e)}\circ\imath_{b G}=b.$$
Indeed, $\pr_{b(e)}\circ\imath_{b G}(g)=\imath_{b G}(g)(b(e))=\tilde\alpha (g, b(e))=b(g)$, $g\in G$. The restriction of $\pr_{b(e)}$ to $\imath_{b G} (G)$ is a homeomorphism. Hence,  $\pr_{b(e)}|_{\mathfrak{E}(b G)}$ is the map of compactifications and $\mathfrak{E}(b G)\geq b G$. 

\medskip

(iii) By the previous $\mathfrak{E}(\mathfrak{E}(b G))\geq\mathfrak{E} (b G)$ and   $\mathfrak{E}(b G)\in\mathbb{E} (G)$. It remains to prove that if $(b G, b)$ is an Ellis compactification of $G$, then $b G\geq\mathfrak{E}(b G)$. 

Let $\tilde\alpha: G\times bG\to bG$ be the extension of action $\alpha$, $\imath_{b G}: G\to (bG)^{bG}$ is an embedding. 

Since $(bG, \bullet)$ is a right topological monoid and continuous multiplication on the right is uniformly continuous on the compact space $b G$ (with the unique uniformity $\tilde{\mathcal U}$), from 
$$\imath_{b G} (g)(t)=\tilde\alpha (g, t)=b (g)\bullet t,\ t\in bG,\ g\in G,$$
it follows that for any $t\in bG$ and   for any ${\rm U}\in \tilde{\mathcal U}$ there exists ${\rm V}\in \tilde{\mathcal U}$ such that 
$$\mbox{if for}\ f, g\in G,\ (f, g)\in {\rm V},\ \mbox{then}\ (b (f)\bullet t, b (g)\bullet t)\in {\rm U}.$$
Therefore, for any entourage $\hat {\rm U}=\{ (f, g)\in G\times G\ |\ (\imath_{b G} (f)(t)=b (f)\bullet t, \imath_{b G} (g)(t)=b (g)\bullet t)\in {\rm U}\}$, where $t\in bG$, ${\rm U}\in\tilde{\mathcal U}$ (from the subbase of the subspace uniformty on $\imath_{b G}(G)$ as a subspace of $b G^{b G}$), there exists  ${\rm V}\in\tilde{\mathcal U}$ such that 
$$\mbox{if}\ (f, g)\in {\rm V},\ \mbox{then}\  (f, g)\in \hat {\rm U}$$
and the map $\imath_{b G}: G\to\imath_{b G} (G)$ is uniformly continuous. Therefore,  $\imath_{b G}$ can be extended to the map of compactifications $\tilde\imath_{b G}: b G\to\mathfrak{E}(b G)$. Hence, $b G\geq\mathfrak{E}(b G)$ and $\mathfrak{E}\circ\mathfrak{E}=\mathfrak{E}$.

\medskip

(i) $\mathfrak{E}$ is an epimorphism. If $b G\in\mathbb{E}(G)$, then $b G=\mathfrak{E}(b G)$. Indeed, by (i) $b G\leq\mathfrak{E}(b G)$, by (ii) $b G\geq\mathfrak{E}(b G)$. Hence, $b G=\mathfrak{E}(b G)$. From this it follows that $\mathfrak{E}$ is an epimorphism. 
\end{proof}

\begin{cor}\label{idemcoin}
$b G\in\mathbb E (G)$ iff $\mathfrak{E}(b G)=b G$.
\end{cor}

\begin{rem}
{\rm The extended variant of Theorem~\ref{thm1} is in~\cite[Theorem 2.22]{KozlovSorin}.}
\end{rem}


\section{$G^*$-compactifications as subsets of spaces of closed binary relations}

\subsection{Hyperspace $2^{X\times X}$ as a space of closed binary relations}

On the product $X\times X$ the symmetry $s_X(x, y)=(y, x)$ is a self-inverse map.

\begin{pro}\label{actionXX}
Let $(G, X, \theta)$ be a $G$-space. 
$$\theta_{\uparrow}: G\times (X\times X)\to X\times X,\ \theta_{\uparrow} (g, (x, y))=(x, \theta(g, y)),$$ 
$$\theta_{\rightarrow}: G\times (X\times X)\to X\times X,\ \theta_{\rightarrow} (g, (x, y))=(\theta (g, x), y),\ g\in G,\ (x, y)\in X\times X,$$
are {\rm(}continuous{\rm)} actions on $X\times X$, $\theta_{\rightarrow}=\theta_{\uparrow}^*$. 
\end{pro}

\begin{proof}
It is easy to verify that the actions are correctly defined. Their continuity (in product topology) is evident. 
$$\theta_{\uparrow}^*(g, (x, y))=s_X(\theta_{\uparrow}(g, s_X((x, y)))=s_X(\theta_{\uparrow}(g, (y, x))=s_X((y, \theta (g, x)))=$$
$$=(\theta (g, x), y)=\theta_{\rightarrow}(x, y),\ g\in G,\ (x, y)\in X\times X.\ \mbox{Hence},\ \theta_{\rightarrow}=\theta_{\uparrow}^*.$$
\end{proof}

A {\it relation on $X$} is a subset of $X\times X$.  If $R, S$ are relations on $X$, then the {\it composition} of $R$ and $S$ (from the right) is the relation $RS=\{(x,y)\ |\ \exists z\ ((x,z)\in S\ \&\ (z,y)\in R)\}$. The symmetry $s_X$ on $X\times X$ induces the self-inverse map $H_{s_X}$ on relations on $X$
$$H_{s_X}: R\to s_X(R)=\{(x, y)\in X\times X\ |\ (y, x)\in R\}.\eqno{(*)}$$ 
The set of all relations on $X$ is a monoid (identity is the relation $\Delta_X=\{(x, x)\ |\ x\in X\}$) with involution $H_{s_X}$~\cite{CliffPrest}. 

A {\it closed relation} on a topological space $X$ is a closed subset of $X\times X$, and the set of all closed relations (including the empty set) can be identified with $2^{X\times X}$. 

From Propositions~\ref{actionXX} and~\ref{actionhyper}, \ref{hyperinversemap} it follows. 

\begin{cor}\label{actionXXYY}
If $(G, X, \theta)$ is a compact $G$-space, then the induced actions 
$$\theta_{\uparrow}^H  (g, R)=\{\theta_{\uparrow}(g, (x, y))=(x, \theta (g, y))\ |\ (x, y)\in R\}\ \mbox{and}\eqno{(\theta_{\uparrow}^H)}$$ 
$$\theta_{\rightarrow}^H (g, R)=\{\theta_{\rightarrow}(g, (x, y))=(\theta (g, x), y)\ |\ (x, y)\in R\},\ R\in 2^{X\times X}, \eqno{(\theta_{\rightarrow}^H)}$$ 
and the self-inverse map $H_{s_X}$ on $(2^{X\times X}, \tau_V)$ are continuous and $(\theta_{\uparrow}^H)^*=\theta_{\rightarrow}^H$.
\end{cor}

\begin{rem}\label{semrel}
{\rm For a metrizable compact space $X$ the continuity of the actions $\theta_{\uparrow}^H$, $\theta_{\rightarrow}^H$ and the self-inverse map $H_{s_X}$ is proved in~\cite{Kennedy}.

If $X$ is a compact space, then $(2^{X\times X}, \tau_V)$ is compact and monoid with continuous involution $H_{s_X}$~\cite{usp2001}. 

\medskip

In general, the composition $(R, S)\to RS$ in $(2^{X\times X}, \tau_V)$ is not continuous on the left (and, hence, on the right). Indeed, let $X=[0, 1]$, $R=([0, \frac12]\times\{0\})\cup ([\frac12, 1]\times\{1\})\subset X\times X$, $S=\{0\}\times [0, \frac12]\subset X\times X$ be closed relations on $X$. Then $RS=\{(0, 0), (0, 1)\}\subset X\times X$. 

Fix $R$ and let $W=[V_1, V_2]$ be a nbd of $RS$, where $V_1=X\times [0, \frac12)$, $V_2=X\times (\frac12, 1]$. Any nbd of $S$ contains the set $S'=\{0\}\times [0, a]$ for some $a<\frac12$, and $RS'=\{(0, 0)\}$. $RS'\not\in W$ and composition is not continuous on the left. }
\end{rem} 

Below $2^{X\times X}=(2^{X\times X}, \tau_V)$. 


\subsection{Maps of hyperspaces $2^{X\times X}$}

\begin{pro}\label{mapshyper} 
{\rm (A)} If $X$, $Y$ are compact spaces and $\varphi: X\to Y$ is an onto map, then 
$$H_{\Phi}: 2^{X\times X}\to 2^{Y\times Y},\  H_{\Phi}(R)=(\Phi=\varphi\times\varphi) (R),\  R\in 2^{X\times X},$$ 
is an onto, perfect map, commutes with self-maps $H_{s_X}$ and $H_{s_Y}$ and homomorphism of monoids with involutions $H_{s_X}$ and $H_{s_Y}$.

\medskip

{\rm (B)} If $(G, X, \theta_X)$, $(G, Y, \theta_Y)$ are compact $G$-spaces and $\varphi: X\to Y$ is an onto $G$-map, then $H_{\Phi}$ is a $G$-map of $G$-spaces $(G, 2^{X\times X},  \theta_{X\uparrow}^{H})$, $(G, 2^{Y\times Y}, \theta_{Y\uparrow}^{H})$ and $G$-spaces $(G, 2^{X\times X},  \theta_{X\rightarrow}^{H})$, $(G, 2^{Y\times Y}, \theta_{Y\rightarrow}^{H})$. 
\end{pro}

\begin{proof}
The first statement of (A) follows from  Proposition~\ref{maphyperF}. $2^{X\times X}$ and $2^{Y\times Y}$ are monoids with involutions $H_{s_X}$ and $H_{s_Y}$. For $R, S\in 2^{X\times X}$ 
$$H_{\Phi} (RS)=\Phi (RS)=\{(\varphi (x), \varphi (y))\ |\ \exists z\ ((x,z)\in S\ \mbox{and}\ (z,y)\in R)\}=$$
$$=\{(\varphi (x), \varphi (y))\ |\ \exists \varphi (z)\ ((\varphi (x), \varphi (z))\in\Phi (S)\ \mbox{and}\ (\varphi (z), \varphi (y))\in\Phi (R))\}=\Phi (R)\Phi (S)=H_{\Phi}(R)H_{\Phi}(S)$$
and $H_{\Phi}(\Delta_X)=\Delta_Y$. Hence, $H_{\Phi}$ is a homomorphism of monoids with involutions.

(B) follows from Proposition~\ref{maphyperF} and Corollary~\ref{actionXXYY}. 
\end{proof}


\subsection{Embedding of a group in the hyperspace $2^{X\times X}$} 

Let $G$ be a group of homeomorphisms of a compact space $X$. The injective map 
$$\i_{\Gamma}: G\to 2^{X\times X},\ \i_{\Gamma}(g)=\Gamma (g)=\{(x, g(x))\ |\ x\in X\}\ (\mbox{\it graph of}\ g),\ g\in G,$$ 
is defined. On $\i_{\Gamma}(G)$, multiplication is the restriction of composition in $2^{X\times X}$, involution is the restriction of the self-inverse map $H_{s_X}$.  

\begin{pro}\label{homomF}
If $(G=(G, \tau_{co}), X, \theta)$ is a compact $G$-space, then the map $\i_{\Gamma}: G\to 2^{X\times X}$ is
\begin{itemize}
\item[(a)]  a $G$-map of $G$-spaces $(G, G, \alpha)$ and $(G, 2^{X\times X}, \theta^H_{\uparrow})$ and commutes with involution $*$ on $G$ and self inverse map $H_{s_X}$ on $2^{X\times X}$,
\item[(b)] topological isomorphism of groups $G$ and $\i_{\Gamma}(G)$,
\item[(c)] uniformly continuous with respect to the Roelcke uniformity $L\wedge R$ on $G$ and the unique uniformity on compact space $2^{X\times X}$.
\end{itemize}
\end{pro}

\begin{proof} 
(a) Let us verify that $\i_{\Gamma}$ is a $G$-map. 
$$\i_{\Gamma} (\alpha (h, g))=\i_{\Gamma} (hg)=\Gamma (hg)=\{(x, (hg)(x))\ |\ x\in X\}=\{(x, \theta (hg, x))\ |\ x\in X\}=$$
$$=\{(x, \theta (h, g(x))\ |\ x\in X\}\stackrel{(\theta^H_{\uparrow})}{=}\theta^H_{\uparrow} (h, \Gamma (g))=\theta^H_{\uparrow} (h, \i_{\Gamma} (g)),\ h, g\in G.$$

Let us verify that $\i_{\Gamma}$ commutes with self-inverse maps $*$ and $H_{s_X}$.
$$\i_{\Gamma} (g^{-1})=\Gamma (g^{-1})=\{(x, g^{-1}(x))\ |\ x\in X\}\stackrel{y=g^{-1}(x)}{=}\{(g(y), y)\ |\ y\in X\}=$$
$$=H_{s_X}(\{(y, g(y))\ |\ y\in X\})=H_{s_X}(\Gamma (g))=H_{s_X}(\i_{\Gamma} (g)),\ g\in G.$$

(b)  
$$\i_{\Gamma} (hg)=\Gamma (hg)=\{(x, (hg)(x))\ |\ x\in X\}=\{(x,  h(g(x)))\ |\ x\in X\}=$$
$$=\{(g(x), h(g(x)))\ |\ x\in X\}\{(x, g(x))\ |\ x\in X\}=\Gamma (h)\Gamma (g)=\i_{\Gamma} (h)\i_{\Gamma} (g),\ h, g\in G,$$
$\i_{\Gamma}(e)=\Delta_X$ is the identity in $2^{X\times X}$. Hence, $\i_{\Gamma}$ is an isomorphism of $G$ and $\i_{\Gamma}(G)$.

Let $\mathcal U$ be the unique uniformity on compact space $X$, the Cartesian product $\mathcal U^2$ of the uniformity $\mathcal U$ is the unique uniformity on compact space $X\times X$. The base of $\mathcal U^2$ is formed by uniform covers 
$$u^2=\{U\times U'\ |\ U,\ U'\in u\}, u\in\mathcal U.$$

The base of the unique uniformity $2^{\mathcal U^2}$ on compact space $(\CL (X\times X), \tau_V)$ (we consider $(\CL (X\times X), \tau_V)$ since $\{\emptyset\}$ is an isolated point in $2^{X\times X}$) is formed by the entourages 
$$\CL^{u^2}=\{(F, T)\in \CL (X\times X)\times \CL (X\times X)\ |\ F\subset \st (T, u^2),\  T\subset \st (F, u^2)\},\ u\in\mathcal U.$$

The compact-open topology on $G$ coincides with the topology induced by the uniformity of uniform  convergence~\cite[Ch.\ 8, \S\ 8.2, Corollary 8.2.7]{Engelking}, which base consists of the entourages of the diagonal 
$$\hat {\rm U}=\{(h, g)\in G\times G\ |\ \forall x\in X\ h(x)\in\st (g(x), u)\},\ \mbox {where}\ u\in\mathcal U.$$

If $(h, g)\in \hat {\rm U}$, then $(\i_{\Gamma} (h)=\Gamma (h), \i_{\Gamma}(g)=\Gamma (g))\in\CL^{u^2}$, $u\in\mathcal U$. Therefore, the map $\i_{\Gamma}$ is continuous.

For $w\in\mathcal U$ let $u\in\mathcal U$ be such that $\{\st (\st (x, u), u)\ |\ x\in X\}\succ w$. 

Any $h\in G$ is a uniformly continuous map of $X$. Hence, for $u$ $\exists\ v\in\mathcal U$, $v\succ u$, such that if $x'\in\st (x, v)$, then $h(x')\in\st (h (x), u)$, $x\in X$. If $(\Gamma (h), \Gamma (g))\in\CL^{v^2}$, then $\forall\ x\in X$ $\exists\ x'\in X$ such that $x'\in\st (x, v)$ and $h(x')\in\st (g(x), v)$. Since $h(x')\in (h(x), u)$, $g(x)\in\st (\st (h(x), u), u)$, $\forall\ x\in X$. Hence, $g(x)\in \st (h(x), w)$, $x\in X$. The restriction $\i_{\Gamma}^{-1}|_{\i_{\Gamma} (G)}$ is continuous at the point $\Gamma (h)$ by Lemma~\ref{cont} and $\i_{\Gamma}$ is a homeomorphic embedding. 

\medskip

(c) $\i_{\Gamma}$ is uniformly continuous with respect to the Roelcke uniformity $L\wedge R$ on $G$ and the unique uniformity on compact space $2^{X\times X}$ by Corollary~\ref{cor1}. 
\end{proof}

\begin{rem}\label{remark}
{\rm  (a) For a compact metrizable space items (a) and (b) of Proposition~\ref{homomF} are proved in~\cite{Kennedy}. In~\cite{usp2001} items (b) and (c) of Proposition~\ref{homomF} are proved.

\medskip

(b) The subspace uniformity on $\i_{\Gamma}(G)$ as a subset of $2^{X\times X}$ ($X$ is a compact space with uniformity $\mathcal U$) has the base consisting  of entourages of the diagonal (reformulation of definition in \S~\ref{hyper})
$$\begin{array}{c}
{\rm U}_{X}=\{(h, g)\in G\times G\ |\ \forall\ y\in X\ \exists\ x\in X\ \mbox{such that}\ y\in\st (x, u),\ h(y)\in\st (g(x), u) \\
\&\ \forall\ x\in X\ \exists\ y\in X\ \mbox{such that}\ x\in\st (y, u),\ g(x)\in\st (h(y), u)\},\ \mbox{where}\ u\in\mathcal U. \\ 
\end{array}\eqno{({\rm U}_X)}$$}
\end{rem}


\subsection{$G^*$-compactifications as subsets of $2^{b G\times bG}$}

Let $G$ be a topological group, $b G\in\mathbb{G} (G)$. By Proposition~\ref{topcoex}  $(G=(G, \tau_{co}), b G, \tilde\alpha)$ is a $G$-space. By Proposition~\ref{homomF} an embedding $\i^{b G}_{\Gamma}: G\to 2^{b G\times b G}$ is defined. Since $\i^{b G}_{\Gamma}  (G)$ is an invariant subset of $G$-space $(G, 2^{b G\times b G},  {\tilde\alpha}^H_{\uparrow})$ and invariant with respect to the self-inverse map $H_{s_{b G}}$ on $2^{b G\times b G}$ (and, hence, an invariant subset of a $G$-space $(G, 2^{b G\times b G}, {\tilde\alpha}^H_{\rightarrow})$), its closure $\cl (\i^{b G}_{\Gamma} (G))$ in $2^{b G\times b G}$ is an invariant subset for the action ${\tilde\alpha}^H_{\uparrow}$ and invariant with respect to the self-inverse map $H_{s_{b G}}$ (and, hence, an invariant subset for the action ${\tilde\alpha}^H_{\rightarrow}$). Therefore, $\cl (\i^{b G}_{\Gamma} (G))$ is a $G^*$-compactification of $G$. Put 
$$\mathfrak{G}: \mathbb{G} (G)\to \mathbb{G}^* (G),\ \mathfrak{G}(b G)=\cl (\i^{b G}_{\Gamma} (G)).$$

\begin{thm}\label{thm2} 
The map $\mathfrak{G}:\mathbb{G}(G)\to \mathbb{G}^* (G)$ has the following properties 
\begin{itemize}
\item $\mathfrak{G}$ is a morphism {\rm(}of posets{\rm)}, 
\item $\forall\ b G\in \mathbb{G}^{\leftrightarrow} (G)$ \quad $b G\leq\mathfrak{G}(b G)$, 
\item $\forall\ b G\in \mathbb{G} (G)$ \quad $\mathfrak{G}(b G)\leq b_r G$,\quad if $b_r G\leq b G$, then $\mathfrak{G}(b G)=b_r G$.
\end{itemize}
Let $b G\in\mathbb{G}^{\leftrightarrow}(G)$ {\rm(}$\mathcal U$ is the subspace uniformity on $G$ as subset of the compact space $b G${\rm)}. 
$\mathfrak{G}(b G)=b G$ iff $\forall\ u\in \mathcal U$ $\exists\ v\in\mathcal U$ such that if $g\in\st (h, v)$, then 
$$\forall\ y\in G\ \exists\ x\in G\ \mbox{such that}\ y\in\st (x, u),\ h(y)\in\st (g(x), u)$$
$$\&\ \forall\ x\in G\ \exists\ y\in G\ \mbox{such that}\ x\in\st (y, u),\ g(x)\in\st (h(y), u).\eqno{\rm(F)}$$
\end{thm}

\begin{proof}
(i) Let $b G, b' G\in\mathbb{G} (G)$, $b G\leq b' G$, and $\varphi: b' G\to b G$ is the map of compactifications of $G$. Then $(\cl (\i^{b' G}_{\Gamma} (G)), \i^{b' G}_{\Gamma})$, $(\cl (\i^{b G}_{\Gamma} (G)), \i^{b G}_{\Gamma})$ are $G^*$-compactifications of $G$ and $H_{\phi}|_{\cl (\i^{b' G}_{\Gamma} (G))}: \cl (\i^{b' G}_{\Gamma} (G))\to \cl (\i^{b G}_{\Gamma} (G))$ is the map of compactifications. Indeed, 
$$H_{\phi} (\i^{b' G}_{\Gamma}(g))=H_{\phi} (\{(x, g(x))\ |\ x\in b'G\})=\{(\varphi (x), \varphi (\tilde\alpha_{b' G} (g, x)))\ |\ x\in b' G\}=$$
$$=\{(\varphi (x), \tilde\alpha_{b G} (g, \varphi (x)))\ |\ x\in b'G\}= \{(y, \tilde\alpha_{b G} (g, y))\ |\ y\in b G\}=\i^{b G}_{\Gamma}(g).$$
Hence, $H_{\phi}\circ\i^{b' G}_{\Gamma}=\i^{b G}_{\Gamma}$, $H_{\phi}|_{\cl (\i^{b' G}_{\Gamma}(G))}$ is the map of compactifications and $\mathfrak{G} (b' G)\geq \mathfrak{G} (b G)$.

\medskip 

(ii) Let $b G\in\mathbb{G}^{\leftrightarrow} (G)$ ($(G, b G, \tilde\alpha)$ is a $G$-space), $\mathcal U$ the subspace uniformity on $G$ as subset of the compact space $(b G, \tilde{\mathcal U})$. Since $ b G\in \mathbb{G}^{\leftrightarrow} (G)$ the inverse action $\tilde\alpha^*$ is defined. 

For $\tilde u\in\tilde{\mathcal U}$ take $\tilde v\in\tilde{\mathcal U}$ such that:
$$\{\st (\st (x, \tilde v), \tilde v)\ |\ x\in b G\}\succ\tilde u.$$
Take $\tilde w\in\tilde{\mathcal U}$  and $O=O^{-1}\in N_G (e)$ such that $O\tilde wO=\{\tilde\alpha (O, \tilde\alpha^* (O, \tilde W))\ |\ \tilde W\in\tilde w\}\succ\tilde v$ (it is possible since $\tilde{\mathcal U}$ is an equiuniformity for the action $\tilde\alpha$ and its inverse $\tilde\alpha^*$, using the proof in  Proposition~\ref{unifcontRoelcke}). Refining in $\tilde w$ a uniform cover if needed, we can assume that $\st (e, w)\subset O$, where $w=\tilde w\wedge G$. 

Let $\Gamma (f)\in\st(\Gamma (h), \tilde w\times\tilde w)\ \&\ \Gamma (h)\in\st(\Gamma (f), \tilde w\times\tilde w)$ and $(e, f(e))\in\st ((x, h(x)), \tilde w\times\tilde w)$. We can assume that $x\in G$ and $(e, f(e))\in\st ((x, h(x)), w\times w)$. Then $e\in\st (x, w)\Longleftrightarrow x\in\st (e, w)\subset O$, and $h(x)=hx\in hO$. Let $\tilde V\in\tilde v$ be such that $hO\subset\tilde V$. Then $f(e)=f\in\st (h(x), \tilde w)\subset\st (\tilde V, \tilde v)\subset\st (\st (h, \tilde v), \tilde v)\subset \tilde U\in\tilde u$ and $f\in \st (h, u)$. 

It is shown that $\mathcal U$ is weaker than the subspace uniformity on $G$ as a subset of compact space $2^{b G\times b G}$ and the inequality $b G\leq \mathfrak{G}(b G)$ is proved. 

\medskip

(iii) Since  $\forall\ b G\in \mathbb{G} (G)$  $\mathfrak{G}(bG)\in\mathbb{G}^*(G)$, from Proposition~\ref{blrlck} it follows that $\mathfrak{G}(b G)\leq b_r G$. 

If $b_r G\leq b G$, then by (i) $\mathfrak{G}(b_r G)\leq\mathfrak{G}(b G)\leq b_r G$. Since $b_r G\in\mathbb{G}^*(G)$ by (ii) $b_r G\leq\mathfrak{G}(b_r G)$. Hence, $\mathfrak{G}(b G)=b_r G$. 

\medskip

The proof of the last statement. Necessity. If $b G=\mathfrak{G}(b G)$, then the map $\i_{\Gamma}: (G, \mathcal U)\to 2^{b G\times b G}$ is uniformly continuous and condition $({\rm U}_X)$ of Remark~\ref{remark} implies fulfillment of the condition (F). 

Sufficiency. Since $b G\in\mathbb{G}^{\leftrightarrow}(G)$, by (ii) $b G\leq\mathfrak{G}(b G)$. The condition (F), condition $({\rm U}_X)$ of Remark~\ref{remark} and Remark~\ref{denseent} imply that the map $\i_{\Gamma}: (G, \mathcal U)\to 2^{b G\times b G}$ is uniformly continuous. Hence,  $b G\geq\mathfrak{G}(b G)$ and, finally, $b G=\mathfrak{G}(b G)$.
\end{proof}

\begin{rem}
{\rm If $b G\in\mathbb{G}^{\leftrightarrow}(G)$ and the condition of Theorem~\ref{thm2} is fulfilled, then $b G\in\mathbb{G}^*(G)$. Indeed, since $b G=\mathfrak{G}(b G)$, $b G\in\mathbb{G}^*(G)$.}
\end{rem}


From Theorems~\ref{thm1} and~\ref{thm2} it follows.

\begin{cor}\label{propBcomp}
\begin{itemize}
\item[(a)] Let $b G\in\mathbb{G}(G)$ be such that $b_r G\leq b G$. Then
$$b_r G=\mathfrak{G}(b G)\leq b G\leq\mathfrak{E}(b G),\ b_r G=\mathfrak{G}(b G)\leq\mathfrak{E}(b_r G)\leq\mathfrak{E}(b G).$$

\item[(b)]  $b_r G=\mathfrak{G}(\beta_G G)\leq\beta_G G=\mathfrak{E}(\beta_G G)$. 

If $G$ is a {\rm SIN}-group, then $b_r G=\mathfrak{G}(\beta_G G)=\beta_G G=\mathfrak{E}(\beta_G G)$. 

\item[(c)] If $b_r G\in\mathbb E (G)$, then $\mathfrak{G}(b_r G)=\mathfrak{E}(b_r G)=b_r G$.

\item[(d)] If $b G\in\mathbb{G}^*(G)\cap\mathbb{E}(G)$, then $b G=\mathfrak{E}(b G)\leq\mathfrak{G}(b G)\leq b_r G$.
\end{itemize}
\end{cor}

\begin{que}
What are further properties of the map $\mathfrak{G}$?
\end{que}

\begin{ex}
{\rm (1) If $G$ is a locally compact group, then for the Alexandroff one-point compactification $\alpha G$ the following hold 
$$\alpha G=\mathfrak{G}(\alpha G)=\mathfrak{E}(\alpha G)$$
and $\alpha G$ is a sim-compactification of $G$.

Indeed, since $\alpha G$ is a compact semitopological semigroup with zero element (ideal) $\infty$, it is a sim-compactification of $G$ ($\alpha G$ is a semitopological inverse monoid with continuous inversion which extends involution on $G$, $(\infty)^*=\infty$, see~\cite{Sain}). By Corollary~\ref{idemcoin} $\alpha G=\mathfrak{E}(\alpha G)$. 

If $R\in\cl (\i^{\alpha G}_{\Gamma} (G))=\mathfrak{G}(\alpha G)$, then $R\cap(\{x\}\times\alpha G)\ne\emptyset$, $x\in\alpha G=G\cup\{\infty\}$. Indeed, otherwize $R\in (W=(\alpha G\times\alpha G)\setminus(\{x\}\times\alpha G))^+$, $W^+\cap\i^{\alpha G}_{\Gamma} (G)=\emptyset$ and, hence, $R\not\in\cl (\i^{\alpha G}_{\Gamma} (G))$. Anologuously, if $R\in\cl (\i^{\alpha G}_{\Gamma} (G))=\mathfrak{G}(\alpha G)$, then $R\cap(\alpha G\times \{x\})\ne\emptyset$, $x\in\alpha G=G\cup\{\infty\}$. 

Let $R\in\cl (\i^{\alpha G}_{\Gamma} (G))=\mathfrak{G}(\alpha G)$, $(x, z)\in R\cap (G\times G)$. Take $g$ such that $z=gx$. We shall show that $R=\i^{\alpha G}_{\Gamma} (g)$. 

Assume that $\exists\ y\in G$ and $p\in\alpha G$ such that $(y, p)\in R$, $p\ne gy$. Since the action $\alpha$ is uniformly equicontinuous with respect to the left uniformity $L$ on $G$ and $R\cap (G\times G)$ is a closed subset of $G\times G$ $\exists$ nbds $O, U\in N_G(e)$ and a nbd $V_p$ of $p$ such that if $h\in G$ and $h(Oy)\cap Ugy\ne\emptyset$, then $h(Oy)\cap V_p=\emptyset$. Put $k=x^{-1}y$. 

Take nbd $W=Ox\times Uz$ of $(x, z)$. $R$, $\imath^{\alpha G}_{\Gamma} (g)\in W^-$. If $h\in G$ is such that $hOx\cap Ugx\ne\emptyset$, then 
$$\exists\ t\in O,\ \exists\ q\in U,\ \mbox{such that}\ htx=qgx\ \Longrightarrow\ htxk=qgxk,$$
$$htxk=hty\in hOy,\ qgxk=qgy\in Ugy,\ \Longrightarrow\ hOy\cap Ugy\ne\emptyset.$$
Therefore, $R\in W^-\cap (Oy\times V_p)^-$, however, $\i^{\alpha G}_{\Gamma} (G)\cap (W^-\cap (Oy\times V_p)^-)=\emptyset$. Hence, $R\not\in\cl (\i^{\alpha G}_{\Gamma} (G))$. The obtained contradiction shows that $R=\i^{\alpha G}_{\Gamma} (g)$. 

From the above it follows that 
$$\mathfrak{G}(\alpha G)=\cl (\i^{\alpha G}_{\Gamma} (G))=\i^{\alpha G}_{\Gamma} (G)\cup\{(x, \infty),\ (\infty, x)\ |\ x\in\alpha G\}=\alpha G.$$ 

(2) Let $b\mathbb Z=\{-\infty\}\cup\mathbb Z\cup\{+\infty\}$ be as in Example~\ref{example}. Then $\mathfrak{E}(b\mathbb Z)=\mathfrak{G}(b\mathbb Z)=b\mathbb Z$. 

Indeed, since $b\mathbb Z\in\mathbb E(\mathbb Z)$,  by Corollary~\ref{idemcoin} $\mathfrak{E}(b\mathbb Z)=b\mathbb Z$. Since $b\mathbb Z\in\mathbb G^*(\mathbb Z)$, it is sufficient to show that the condition (F) of Theorem~\ref{thm2} is valid. 

Let $\mathcal U$ be the subspace uniformity on $\mathbb Z$ as a subset of $b\mathbb Z$. Without loss of generality, take 
$$u=\{(-\infty, -n], \{-n+1\}, \ldots, \{n-1\}, [n, +\infty)\}\in\mathcal U,\ n\in\mathbb N,$$
and $v=\{(-\infty, -2n], \{-2n+1\}, \ldots, \{2n-1\}, [2n, +\infty)\}\in\mathcal U$. If $g\in\st (h, v)$, then either 
\begin{itemize}
\item[(a)] $g, h\leq -2n$, or 
\item[(b)] $g, h=k$, $k\in [-2n+1, 2n-1]$, or 
\item[(c)] $g, h\geq 2n$. 
\end{itemize}
In (a) $g(n), h(n)\leq -n$. Hence, for $k\leq n$\ $g(k)\in\st (h(k), u)$. For $k>n$\ $g(k)=h(k+(g-h))$ and $k+(g-h)\in\st (k, u)$, if $g\leq h$ (the case $g\geq h$ is treated similarly). In (b) $g=h$ and (c) is treated as (a). Therefore, the condition (F) of Theorem~\ref{thm2} is valid.}
\end{ex}


\end{document}